\documentclass[11pt,reqno]{amsart}
\usepackage{amsmath, amssymb, amsthm}
\usepackage{url}
\usepackage{hyperref}
\usepackage{tikz}
\usepackage{ytableau}
\usepackage[utf8]{inputenc}
\usepackage{xcolor}
\usepackage{tikz-cd}
\usepackage{hyperref}
\usepackage{nccmath}

\usepackage{youngtab}

\usepackage{array}
\usepackage{booktabs}

\usetikzlibrary{decorations.pathreplacing}

\usetikzlibrary{decorations.pathreplacing}

\newcommand{\s}{{\sigma}}

\renewcommand{\a}{\alpha}
\renewcommand{\b}{\beta}

\renewcommand{\d}{{\delta}}
\newcommand{\g}{\gamma}

\renewcommand{\l}{\lambda}

\renewcommand{\(}{\left\(}
\renewcommand{\)}{\right\)}
\renewcommand{\[}{\left\[}
\renewcommand{\]}{\right\]}

\numberwithin{equation}{section}
\theoremstyle{plain}
\newtheorem{theorem}{Theorem}[section]
\newtheorem{lemma}[theorem]{Lemma}

\newtheorem{problem}[theorem]{Problem}

\newtheorem{proposition}[theorem]{Proposition}

\usepackage{mathrsfs}

\makeatletter
\def\proof{\@ifnextchar[{\@oproof}{\@nproof}}
\def\@oproof[#1][#2]{\trivlist\item[\hskip\labelsep\textit{#2 Proof of\
		#1.}~]\ignorespaces}
\def\@nproof{\trivlist\item[\hskip\labelsep\textit{Proof.}~]\ignorespaces}

\makeatother

\makeatletter
\@namedef{subjclassname@2020}{%
	\textup{2020} Mathematics Subject Classification}
\makeatother
\begin{document}
	
	\title[]{Parity of parts and excludant statistics in partitions} 

	\author{Gargi Mukherjee}
\address{School of Mathematical Sciences,
	National Institute of Science Education and Research, Bhubaneswar, An OCC of Homi Bhabha National Institute,  Via- Jatni, Khurda, Odisha- 752050, India}
\email{gargi@niser.ac.in}

\subjclass[2020]{11P81, 11P82,11P84}
\keywords{Integer partitions, minimal and maximal excludants, Tauberian methods, asymptotics.}
	\maketitle
	\begin{abstract}
In this paper, we study restricted excludant statistics depending on its parity in partitions where parts with same parity are distinct. Using $q$-series transformations, we show that generating functions of these partition statistics are related to the quantum modular forms $\s(q)$, its companion $\sigma^*(q)$ introduced by Ramanujan, and $v_2(q)$, a Nahm-type sum, introduced by Andrews. Utilizing Tauberian method, we obtain asymptotics of such sequences. 
	\end{abstract}
\allowdisplaybreaks
	\section{Introduction and statement of main results}\label{intro}
	A partition of a non-negative integer $n$ is a non-increasing sequence of positive integers $\l=(\l_1,\l_2,\cdots,\l_r)$ with $\sum_{j=1}^{r}\l_j=n$, denoted by $\l\vdash n$. Total number of partitions of $n$ is denoted by $p(n)$ and $p(0):=1$. Due to Euler, the generating function is given by
	 $$\sum_{n=}^{\infty}p(n)q^n=\frac{1}{\left(q\right)_{\infty}},$$
where the $q$-Pochhammer symbol is defined as
$$
\left(a;q\right)_0:=1, \left(a;q\right)_n:=\prod_{k=1}^{n}\left(1-aq^k\right),\ \text{and}\ \left(a;q\right)_{\infty}:=\lim_{n\to \infty}(a;q)_n.
$$
	For $\l\vdash n$, $\ell(\l)$ and $a(\l)$ denote the length of the partition and the largest part of $\l$ respectively. The multiplicity of a part $\l_j$ in $\l\vdash n$ is denoted by $\text{mult}(\l_j)$. Often we write a partition $\l\vdash n$ where each part appears with multiplicity $m_j\in \mathbb{N}_0$ as $\l=\left(\l^{m_1}_1\l^{m_2}_2\cdots\l^{m_r}_{r}\right)$. Let $\mathcal{P}_{od}(n)$ (resp. $\mathcal{P}_{ed}(n)$) be the set of all partitions of $n$ into distinct odd (resp. even) parts and $\left|\mathcal{P}_{od}(n)\right|:=p_{od}(n)$ (resp. $\left|\mathcal{P}_{ed}(n)\right|:=p_{ed}(n)$) with $p_{od}(0)=p_{ed}(0):=1$. For example, $p_{od}(6)=5$ enumerated by $6,\ 5+1,\ 4+2,\ 3+2+1,\ 2+2+2,$ and $p_{ed}(5)=6$ enumerated by
	$5,\ 4+1,\ 3+2,\ 3+1+1,\ 2+1+1+1,\ 1+1+1+1+1+1.$
The corresponding generating functions are given by
$$\sum_{n=0}^{\infty}p_{od}(n)q^n=\frac{\left(-q;q^2\right)_{\infty}}{\left(q^2;q^2\right)_{\infty}}\ \ \text{and}\ \ \sum_{n=0}^{\infty}p_{ev}(n)q^n=\frac{\left(-q^2;q^2\right)_{\infty}}{\left(q;q^2\right)_{\infty}}.$$
Fraenkel and Peled \cite{FP} introduced the {\it minimal excludant} statistics, denoted by {\it mex}, on $S\subset\mathbb{N}$ as the least positive integer which is not in $S$. In the context of theory of partitions, Andrews and Newman \cite{AN} studied the mex statistics\footnote{Grabner and Knopfmacher \cite{GK} studied mex under another terminology called {\it the least gap}.}. Ballantine and Merca \cite{BM} proved an identity of Andrews and Newman \cite[Theorem 1.1]{AN} by purely combinatorial argument and extended this to study partitions with {\it least $r$-gaps}. Hopkins, Sellers, and Stanton \cite{HSS} generalized the concept of mex and showed that these are related to the crank statistics of partitions. Later Hopkins, Sellers, and Yee \cite{HKY} developed a combinatorial framework to find a connection between odd mex and non-negative crank of partitions. Bhoria, Eyyunni, Kaur, and Maji \cite{BEKM} studied mex for partitions into distinct parts. Later Bhoria, Eyyuni, and Maji \cite{BEM} introduced the {\it $r$-chain mex} and studied partition identities in the spirit of Franklin. Recently, Dhar, Mukhopadhyay and Sarma \cite{DMS} studied mex for overpartitions. In contrast to the minimal excludant, Chern \cite{Chern} considered {\it maximal excludant} statistics of partitions, denoted by {\it max}, where $\text{max}(\l)$ is the largest non-negative integer smaller than the largest part of $\l\vdash n$ that is not a part of $\l$.

In this paper, we study both minimal and maximal excludant for partitions without repeated odd parts and minimal excludant for partitions without repeated even parts. For $\l \in \mathcal{P}_{od}(n)$, $\text{mex}(\l)$ is the smallest positive integer that is not a part of $\l$. Depending on the parity of $\text{mex}(\l)$, we define
\begin{equation}\label{def2}
a^{o}_{od}(n):=\hspace{-0.65 cm}\underset{\text{mex}(\l) \equiv 1\left(\text{mod}\ 2\right)}{\underset{\l \in \mathcal{P}_{od}(n)}{\sum}}\hspace{-0.7 cm} 1\ \ \text{and}\ \ 
a^{e}_{od}(n):=\hspace{-0.65 cm}\underset{\text{mex}(\l)\equiv 0\left(\text{mod}\ 2\right)}{\underset{\l \in \mathcal{P}_{od}(n)}{\sum}}\hspace{-0.7 cm} 1.
\end{equation}
In other words, $a^{o}_{od}(n)$ (resp. $a^{e}_{od}(n)$) counts the number of partitions into distinct odd parts with $\text{mex}(\l)$ is odd (resp. even). For $n=6$, we have
\begin{center}
	\begin{tabular}{ |c|c|c|c| } 
		\hline
		$\l \in \mathcal{P}_{od}(6)$ & mex$(\l)$ & mex$(\l)\equiv 1\left(\text{mod}\ 2\right)$& mex$(\l)\equiv 0\left(\text{mod}\ 2\right)$\\ 
		\hline 
		$6$ & $1$& $1$ &-\\ 
		\hline 
		$5+1$ & $2$ & - & $2$\\  
		\hline
		$4+2$ & $1$ & $1$ &-\\
		\hline
		$3+2+1$ & $4$ & - & $4$\\
		\hline
		$2+2+2$ & $1$ & $1$ &-\\
		\hline
	\end{tabular}
\end{center}
Therefore, $a^{o}_{od}(6)=3$ enumerated by $6, 4+2, 2+2+2$ and $a^{e}_{od}(6)=2$ enumerated by $5+1, 3+2+1$. To state our first result, we recall Andrews' \cite[Equation (5)]{Andconj} $v_2(q)$ function related to parity (modulo $4$) of rank of partitions
\begin{equation*}
v_2(q):=\sum_{n=0}^{\infty}\frac{q^{2n^2-n}}{(-q;q^2)_n}.
\end{equation*}
\begin{theorem}\label{thm_aod}  We have
	\begin{equation}\label{thm_aod1}
	\sum_{n=0}^{\infty}a^{o}_{od}(n)q^n=\frac{(-q;q^2)_{\infty}}{(q^2;q^2)_{\infty}} \sum_{n=0}^{\infty} \frac{q^{2n^2+n}}{(-q;q^2)_{n+1}}
	\end{equation} and 
	\begin{equation}\label{thm_aod2}
	\sum_{n=0}^{\infty}a^{e}_{od}(n)q^n=\frac{(-q;q^2)_{\infty}}{(q^2;q^2)_{\infty}} \left(v_2(q)-\sum_{n=0}^{\infty}\frac{q^{2n^2+n}}{(-q;q^2)_{n}}\right).
	\end{equation}
\end{theorem}
The monotonicity of $\{a^{o}_{od}(n)\}_{n\ge 0}$ and $\{a^{e}_{od}(n)\}_{n\ge 0}$ are stated in the following theorem.
\begin{theorem}\label{increasing2}
	For $n \ge 1$, we have 
	\begin{equation}\label{aood}
	{a}^{o}_{od}(n+1)>{a}^{o}_{od}(n),
	\end{equation}
	and for $n \ge 7$, 
	\begin{equation}\label{aeod}
	{a}^{e}_{od}(n+1)>{a}^{e}_{od}(n).
	\end{equation}
\end{theorem}
Next, we present the asymptotics of ${a}^{o}_{od}(n)$ and ${a}^{e}_{od}(n)$.
\begin{theorem}\label{asymp2}
	We have, as $n\to \infty$,
	\begin{equation}\label{asympaodoe}
	a^{o}_{od}(n) \sim \frac{e^{\pi\sqrt{\frac n2}}}{4\sqrt{2}n}\ \ \text{and} \ \ a^{e}_{od}(n) \sim \frac{e^{\frac{\pi}{2}\sqrt{n}}}{\sqrt{2}n}.
	\end{equation}
\end{theorem}
Now, we study partitions without repeated odd parts with respect to the parity of minimal excludants. For $\l \in \mathcal{P}_{od}(n)$, $\text{moex}(\l)$ (resp. $\text{meex}(\l)$) is the smallest odd (resp. even) integer which is not a part of $\pi$ and define
\begin{equation}\label{def3}
\sigma_{od}\text{moex}(n):=\underset{\l \in \mathcal{P}_{od}(n)}{\sum}\text{moex}(\l)\ \ \text{and}\ \ \sigma_{od}\text{meex}(n):=\underset{\l \in \mathcal{P}_{od}(n)}{\sum}\text{meex}(\l).
\end{equation}
For example, with $n=6$, we have
\begin{center}
	\begin{tabular}{ |c|c|c| } 
		\hline
		$\l \in \mathcal{P}_{od}(6)$ & moex$(\l)$ & meex$(\l)$\\ 
		\hline 
		$6$ & $1$& $2$ \\ 
		\hline 
		$5+1$ & $3$ & $2$\\  
		\hline
		$4+2$ & $1$ & $6$\\
		\hline
		$3+2+1$ & $5$ & $4$\\
		\hline
		$2+2+2$ & $1$ & $4$\\
		\hline
		Total & $\s_{od}\text{moex}(6)=11$ & $\s_{od}\text{meex}(6)=18$\\
		\hline 
	\end{tabular}
\end{center}
To state the results in this context, let us we recall $\sigma^*(q)$, studied by Cohen \cite[Equation (2)]{Cohen},
\begin{equation}\label{cohensigma}
\s^*(q):=2\sum_{n=1}^{\infty}\frac{(-1)^nq^{n^2}}{(q;q^2)_n}.
\end{equation}
The following theorem relates the generating function of $\sigma_{od}\text{moex}(n)$ (resp. $\sigma_{od}\text{meex}(n)$) to $\s^*(q)$ (resp. to a modular form).
\begin{theorem}\label{thm_sigma}
We have
	\begin{equation}\label{thm_sigma1}
	\sum_{m=0}^{\infty}\sigma_{od}\textnormal{moex}(n)q^n= \frac{(-q;q^2)_{\infty}}{(q^2;q^2)_{\infty}}\Bigl(1+\s^*(-q)\Bigr)
	\end{equation} and 
	\begin{equation}\label{thm_sigma2}
	\sum_{m=0}^{\infty}\sigma_{od}\textnormal{meex}(n)q^n=2 \sum_{n=0}^{\infty}p_{ed}(n)q^n.
	\end{equation}
\end{theorem}
As an immediate consequence of Theorem \ref{thm_sigma}, we have the following identity.
\begin{theorem}\label{rem2}
	For $n\in \mathbb{N}$, we have $\sigma_{od}\textnormal{meex}(n)=2 p_{ed}(n)$.
\end{theorem}
\allowdisplaybreaks
\noindent For instance, for $n=5$, we have 
\begin{table}[ht]
	\centering
	\begin{tabular}{|l|c|c|c|c|c|}
		\hline
		$\l\in\mathcal{P}_{od}(5)$ & $5$ & $4+1$ & $3+2$ & $2+2+1$ & Total\\  \hline
		$\text{meex}(\l)$ & $2$ & $2$ & $4$ & $4$ & $12$ \\ \hline
	\end{tabular}
\end{table}
and $p_{ed}(5)=6$ enumerated by $5, 4+1, 3+2, 3+1+1, 2+1+1+1, 1+1+1+1+1+1$. Thus $\s_{od}\text{meex}(5)=12=2 p_{ed}(5)$. The monotonicity of $\{\sigma_{od}\text{moex}(n)\}_{n\ge 0}$ is stated in the following theorem.
\begin{theorem}\label{increasing}
	For $n \ge 4$, we have $\sigma_{od}\textnormal{moex}(n+1)>\sigma_{od}\textnormal{moex}(n)$.
\end{theorem}
\noindent Our next result is on the asymptotics of $\sigma_{od}\text{moex}(n)$.
\begin{theorem}\label{asymp}
	We have, as $n\to \infty$,
	\begin{equation*}
	\sigma_{od}\textnormal{moex}(n) \sim \frac{3e^{\pi\sqrt{\frac{n}{2}}}}{4\sqrt{2}n}.
	\end{equation*}
\end{theorem}
The following identity relates $\textnormal{meex}$ statistics and divisibility of even parts modulo $4$.
\begin{theorem}\label{thm_meex-4e}
	Let $p_{od,e}(n,j)$ be the number of paritions into distinct odd part of $n$ each of which has exactly $j$ parts greater than their meex and $p_{4,e}(n,j)$ be the number of partitions of $n$ with only multiples of $4$ appear as even parts and has exactly $j$ even parts. Then 
	$p_{od,e}(n,j)=p_{4,e}(n,j)$.
\end{theorem}

Now, we consider maximal excludant for partitions without repeated odd parts. For $\l \in \mathcal{P}_{od}(n)$, $max(\l)$ is the largest non-negative integer, smaller than the largest part of $\l$ that is not a part of $\l$. Define
\begin{equation}\label{def6}
\overline{a}^{o}_{od}(n):=\hspace{-0.65 cm} \underset{\text{max}(\l) \equiv 1\ \left(\text{mod}\ 2\right)}{\underset{\l \in \mathcal{P}_{od}(n)}{\sum}}\hspace{-0.4 cm} 1\ \ \text{and}\ \ \overline{a}^{e}_{od}(n):=\hspace{-0.65 cm}\underset{\text{max}(\l)\equiv 0\ \left(\text{mod}\ 2\right)}{\underset{\l \in \mathcal{P}_{od}(n)}{\sum}}\hspace{-0.4 cm} 1.
\end{equation}
%
For $n=6$, 
\begin{center}
	\begin{tabular}{ |c|c|c|c| } 
		\hline
		$\l \in \mathcal{P}_{od}(6)$ & max$(\l)$ & max$(\l)\equiv 1\left(\text{mod}\ 2\right)$& max$(\l)\equiv 0\left(\text{mod}\ 2\right)$\\ 
		\hline 
		$6$ & $5$& $5$ &-\\ 
		\hline 
		$5+1$ & $4$ & - & $4$\\  
		\hline
		$4+2$ & $3$ & $3$ &-\\
		\hline
		$3+2+1$ & $0$ &-& $0$\\
		\hline
		$2+2+2$ & $1$ & $1$ &-\\
		\hline
	\end{tabular}
\end{center}
Therefore, $\overline{a}^{o}_{od}(6)=3$ enumerated by $6, 4+2, 2+2+2$ and $\overline{a}^{e}_{od}(6)=2$ enumerated by $5+1, 3+2+1$. Recall Ramanujan's third order mock theta function \cite[Equation (2.1.6)]{AB}
$$\nu(q):=\sum_{n=0}^{\infty}\frac{q^{n(n+1)}}{\left(-q;q^2\right)_{n+1}}.$$
\begin{theorem}\label{thm_abarod} We have
	\begin{equation}\label{thm_abarod1}
	\sum_{n=0}^{\infty}\overline{a}^{o}_{od}(n)q^n=\frac{q}{(q^3;q^2)_{\infty}}\left(\nu(-q)-\frac{1}{1-q}\right)+\frac{q^2(-q;q^2)_{\infty}}{(1+q)(q^4;q^2)_{\infty}}\sum_{m\ge 0}\frac{q^{3m}}{1+q^{2m+1}}
	\end{equation} and
	\begingroup
	\allowdisplaybreaks 
	\begin{align}\label{thm_abarod2}
	\sum_{n=0}^{\infty}\overline{a}^{e}_{od}(n)q^n&=\sum_{m=0}^{\infty}(-q,q)_{m}q^{m+1}+\frac{(-q;q^2)_{\infty}}{(q^2;q^2)_{\infty}}\sum_{m=1}^{\infty}\frac{(-q;q)_{m-1}(q^{2};q^2)_{m-1}}{(-q;q^2)_{m}}q^{3m}\nonumber\\
	&\hspace{4 cm}+\frac{(-q;q^2)_{\infty}}{(q^2;q^2)_{\infty}}\sum_{m=1}^{\infty}\frac{(-q;q)_{m-1}(q^{2};q^2)_{m}}{(-q;q^2)_{m+1}}q^{3m+1}.
	\end{align}
	\endgroup
\end{theorem}
Analogous to the parity of minimal excludants, here we consider parity of maximal exculdant statistics for partitions with odd parts distinct. For $\l \in \mathcal{P}_{od}(n)$, $moax(\l)$ (resp. $meax(\l)$) denote the largest odd integer (resp. the largest even integer) smaller than the largest part of $\pi$ which is not a part of $\l$. Define
\begin{equation}\label{def7}
\sigma_{od}\text{moax}(n):=\hspace{-0.2 cm}\underset{\l \in \mathcal{P}_{od}(n)}{\sum}moax(\l)\ \ \text{and}\ \ \sigma_{od}\text{meax}(n):=\hspace{-0.2 cm}\underset{\l \in \mathcal{P}_{od}(n)}{\sum}meax(\l). 
\end{equation}
For instance, with $n=6$, we have
\begin{center}
	\begin{tabular}{ |c|c|c| } 
		\hline
		$\pi \in \mathcal{P}_{od}(6)$ & $moax(\l)$ & $meax(\l)$\\ 
		\hline 
		$6$ & $5$& $4$ \\ 
		\hline 
		$5+1$ & $3$ & $4$\\  
		\hline
		$4+2$ & $3$ & $0$\\
		\hline
		$3+2+1$ & $-$ & $0$\\
		\hline
		$2+2+2$ & $1$ & $0$\\
		\hline
		Total & $\s_{od}\text{moax}(6)=12$ & $\s_{od}\text{meax}(6)=8$\\
		\hline 
	\end{tabular}
\end{center}
The generating functions of $\sigma_{od}\text{moax}(n)$ and $\sigma_{od}\text{meax}(n)$ are given below. 
\begin{theorem}\label{thm_sigmabar}
	We have
	\begin{equation}\label{thm_sigmabar1}
	\sum_{m=0}^{\infty}\sigma_{od}\textnormal{moax}(n)q^n=\frac{1}{(q^2;q^2)_{\infty}}\sum_{k=1}^{\infty}(2k-1)(-q;q^2)_{k-1}\sum_{m=1}^{\infty}q^{m(m+2k)},
	\end{equation} and 
	\begin{equation}\label{thm_sigmabar2}
	\sum_{m=0}^{\infty}\sigma_{od}\textnormal{meax}(n)q^n=2\frac{(-q;q^2)_{\infty}}{(q^2;q^2)_{\infty}} \Biggl(1-(q^4;q^4)_{\infty}+\sum_{n=1}^{\infty}\frac{q^{2(n+1)}}{1-q^{2(n+1)}}\Bigl(1-(q^4;q^4)_n\Bigr)\Biggr).
	\end{equation}
\end{theorem}
Now, we study minimal excludants for paritions without repeated even parts. For $\l \in \mathcal{P}_{ed}(n)$, define
\begin{equation}\label{evneq1}
a^{o}_{ed}(n):=\hspace{-1 cm}\underset{mex(\l)\equiv 1\ (\text{mod}\ 2)}{\underset{\l\in \mathcal{P}_{ed}(n)}{\sum}}\hspace{-0.5 cm}1\ \ \ \text{and}\ \ \ a^{e}_{ed}(n):=\hspace{-1 cm}\underset{mex(\l)\equiv 0\ (\text{mod}\ 2)}{\underset{\l\in \mathcal{P}_{ed}(n)}{\sum}}\hspace{-0.5 cm}1.
\end{equation}
Note that $a^{o}_{ed}(n)$ (resp. $a^{o}_{ed}(n)$) counts the number of  partitions of $n$ into distinct even parts with $mex$ odd (resp. even). For instance, $a^{o}_{ed}(6)=4$ enumerated by 
$6,\ 4+2,\ 3+3,\ 2+1+1+1+1$
and $a^{e}_{ed}(6)=5$ enumerated by $5+1,\ 4+1+1,\ 3+2+1,\ 3+1+1+1,\ 1+1+1+1+1+1$. The corresponding generating functions are as follows.
\begin{theorem}\label{evthm1}
We have
$$\sum_{n=0}^{\infty}a^{o}_{ed}(n)q^n=\frac{\left(-q^2;q^2\right)_{\infty}}{\left(q;q^2\right)_{\infty}}\left(\sum_{n=0}^{\infty}\frac{q^{2n^2+n}}{\left(-q^2;q^2\right)_n}-\sum_{n=0}^{\infty}\frac{q^{(2n+1)(n+1)}}{\left(-q^2;q^2\right)_n}\right),$$
and
$$\sum_{n=0}^{\infty}a^{e}_{ed}(n)q^n=\frac{\left(-q^2;q^2\right)_{\infty}}{\left(q;q^2\right)_{\infty}}\sum_{n=0}^{\infty}\frac{q^{2n^2-n}}{\left(-q^2;q^2\right)_n}.$$
\end{theorem}
Finally, for $\l\in \mathcal{P}_{ed}(n)$, define $moex_{ed}(\l)$ (resp. $meex_{ed}(\l)$) is the smallest positive odd (resp. even) integer that is not a part of $\l$, defined by
\begin{equation}\label{eveqn2}
\s_{ed}\text{moex}(n):=\underset{\l\in \mathcal{P}_{ed}(n)}{\sum}moex_{ed}(\l)\ \ \text{and}\ \ \s_{ed}\text{meex}(n):=\underset{\l\in \mathcal{P}_{ed}(n)}{\sum}meex_{ed}(\l).
\end{equation}
To state our next theorem, recall Ramanujan's $\s(q)$ function (see \cite[Equation (1.1)]{ADH})
\begin{equation}\label{sigma}
\sigma(q):=\sum_{n=0}^{\infty}\frac{q^{\frac{n(n+1)}{2}}}{\left(-q;q\right)_n}.
\end{equation}
\begin{theorem}\label{evthm2}
We have
$$\sum_{n=0}^{\infty}\s_{ed}\textnormal{moex}(n)q^n=\frac{\left(q^2;q^2\right)^2_{\infty}\left(q\right)^3_{\infty}}{\left(q^4;q^4\right)_{\infty}}\ \ \text{and}\ \ \sum_{n=0}^{\infty}\s_{ed}\textnormal{meex}(n)q^n=2\frac{\left(-q^2;q^2\right)_{\infty}}{\left(q;q^2\right)_{\infty}}\s\left(q^2\right).$$	
\end{theorem}
We organize the remaining part of the paper as follows. Section \ref{sec1} presents preliminary results necessary for the later sections. In Section \ref{sec2}, we prove Theorems \ref{thm_aod}--\ref{asymp2}. Next, we present proofs of Theorems \ref{thm_sigma}--\ref{thm_meex-4e} in Section \ref{sec3}. Proof of Theorems \ref{thm_abarod} and \ref{thm_sigmabar} are given in Section \ref{sec4}. In Section \ref{sec5}, we prove Theorems \ref{evthm1} and \ref{evthm2}. Finally, we conclude this paper by stating a few problems in Section \ref{con}.

\section{Preliminary lemmas}\label{sec1}
In this section, we state $q$-series transformations and Tauberian theorem which will be useful to prove our main results.

\noindent Due to Cauchy \cite[Theorem 2.1]{A} (also known as $q$-Binomial theorem), we have the following.
\begin{lemma}\label{Cauchy}
	If $|q|<1$, $|t|<1$, then
	\begin{equation*}
	\sum_{n=0}^{\infty}\frac{(aq;q)_n t^n}{(q;q)_n}=\frac{(at;q)_{\infty}}{(t;q)_{\infty}}.
	\end{equation*}
\end{lemma}
From \cite[Equations (2.2.12) and (2.2.13)]{A}, here we record two identities of Gauss.
\begin{lemma}
We have
\begin{align}\label{Gauss1}
\sum_{n\in \mathbb{Z}}(-1)^n q^{n^2}&=\frac{\left(q;q\right)_{\infty}}{\left(-q;q\right)_{\infty}},\\\label{Gauss}
\sum_{n=0}^{\infty}q^{\frac{n(n+1)}{2}}&=\frac{\left(q^2;q^2\right)_{\infty}}{\left(q;q^2\right)_{\infty}}.	
\end{align}
\end{lemma}
Next, we state Heine's transformation formula \cite[page 359, (III.1)]{Heine}. To state it, define
$${}_{2}\phi_{1}\left(\begin{matrix}
a,&b\\&c
\end{matrix}\,;q,z\right):=\sum_{n=0}^{\infty}\frac{\left(a\right)_n\left(b\right)_n z^n}{\left(q\right)_n\left(c\right)_n}.$$
\begin{lemma}\label{hypg_tran_formula}
\begin{equation*}
{}_{2}\phi_{1}\left(\begin{matrix}
a,&b\\&c
\end{matrix}\,;q,z\right) =\frac{(b)_{\infty}(az)_{\infty}}{(c)_{\infty}(z)_{\infty}} {}_{2}\phi_{1}\left(\begin{matrix}
\frac cb,&z\\&az
\end{matrix}\,;q,b\right).
\end{equation*}	
\end{lemma} 
Due to Andrews' \cite[Theorem 1]{Andrews}, we have the following:
\begin{theorem}\label{APT}
We have
\begin{align*}
\sum_{n=0}^{\infty}\frac{\left(B;q\right)_{n}\left(-Abq;q\right)_nq^n}{\left(-aq;q\right)_n\left(-bq;q\right)_n}&=-\frac{a^{-1}\left(B;q\right)_{\infty}\left(-Abq;q\right)_{\infty}}{\left(-aq;q\right)_{\infty}\left(-bq;q\right)_{\infty}}\sum_{m=0}^{\infty}\frac{\left(A^{-1};q\right)_m\left(\frac{Abq}{a}\right)^m}{\left(-\frac{B}{a};q\right)_{m+1}}\\
&\hspace{1 cm}+\left(1+b\right)\sum_{m=0}^{\infty}\frac{\left(-a^{-1};q\right)_{m+1}\left(-\frac{ABq}{a}\right)_m\left(-b\right)^m}{\left(-\frac{B}{a};q\right)_{m+1}\left(\frac{Abq}{a};q\right)_{m+1}}.
\end{align*}
\end{theorem}
\noindent Next, we state Ingham's \cite{Ingham} Tauberain theorem.
\begin{proposition}\label{ingham}
	Let $A(q)=\sum_{n=0}^{\infty}a(n)q^n$ be a power series with radius of convergence 1. Assume that $\{a(n)\}$ is a weakly increasing sequence of non-negative real numbers. If there are constants $\alpha, \beta \in \mathbb{R}$, and $C>0$ such that \begin{equation*}
	A(e^{-t})\sim \alpha t^{\beta}e^{\frac{C}{t}}\ \ \text{as} \ \ t\rightarrow 0^{+},
	\end{equation*}
	then 
	\begin{equation*}
	a(n)\sim \frac{\alpha}{2\sqrt{\pi}}\frac{C^{\frac{2\beta+1}{4}}e^{2\sqrt{Cn}}}{n^{\frac{2\beta+3}{4}}}\ \ \text{as} \ \ n\rightarrow \infty.
	\end{equation*}
\end{proposition}
We also require the following asymptotics of $\frac{1}{(q)_{\infty}}$ with $q=e^{-t}$ and $t\to 0^{+}$. From \cite[(118.5)]{Rademacher}, we have 
\begin{equation}\label{eqn1}
\frac{1}{(e^{-t};e^{-t})}\sim \sqrt{\frac {t}{2\pi}}e^{\frac{\pi^2}{6t}}\ \ \text{as} \ \ t\rightarrow 0^{+}.
\end{equation}

\section{Proof of Theorems \ref{thm_aod}--\ref{asymp2}}\label{sec2}


\emph{Proof of Theorem \ref{thm_aod}}: From \ref{def2}, we have
\begin{eqnarray}
\sum_{n=0}^{\infty}a^{o}_{od}(n)q^n &=&  \sum_{n=0}^{\infty} \frac{q^{1+2+3+\cdots+(2n-1)+2n}\displaystyle\prod_{m=n+1}^{\infty}(1+q^{2m+1})}{(q^2;q^2)_{\infty}}\nonumber \\ &=& \frac{(-q;q^2)_{\infty}}{(q^2;q^2)_{\infty}} \sum_{n=0}^{\infty} \frac{q^{\binom{2n+1}{2}}}{(-q;q^2)_{n+1}}=\frac{(-q;q^2)_{\infty}}{(q^2;q^2)_{\infty}} \sum_{n=0}^{\infty} \frac{q^{2n^2+n}}{(-q;q^2)_{n+1}}, \nonumber
\end{eqnarray}
which concludes the proof of \eqref{thm_aod1}. 

Next, for $a^{e}_{od}(n)$, according to \ref{def2}, we get
\begin{eqnarray}
&&\sum_{n=0}^{\infty}a^{e}_{od}(n)q^n = \sum_{n=0}^{\infty} \frac{q^{1+2+3+\cdots+(2n-1)}{\displaystyle\prod_{m=n}^{\infty}(1+q^{2m+1})}}{\displaystyle\prod_{\underset{m \neq n}{m=1}}^{\infty}(1-q^{2m})}= \frac{(-q;q^2)_{\infty}}{(q^2;q^2)_{\infty}} \sum_{n=0}^{\infty} \frac{q^{\binom{2n}{2}}(1-q^{2n})}{(-q;q^2)_{n}} \nonumber\\
&=&\frac{(-q;q^2)_{\infty}}{(q^2;q^2)_{\infty}} \left(\sum_{m=0}^{\infty}\frac{q^{2n^2-n}}{(-q;q^2)_n}-\sum_{n=0}^{\infty}\frac{q^{2n^2+n}}{(-q;q^2)_{n}}\right)= \frac{(-q;q^2)_{\infty}}{(q^2;q^2)_{\infty}} \left(v_2(q)-\sum_{n=0}^{\infty}\frac{q^{2n^2+n}}{(-q;q^2)_{n}}\right).\nonumber 
\end{eqnarray}
This concludes the proof.\qed 

\emph{Proof of Theorem \ref{increasing2}}: Recall that $\mathcal{P}_{od}(n)$ denotes the number of partitions of $n$ without repeated odd parts, precisely,
\begin{equation}\label{oddptn1}
\mathcal{P}_{od}(n):=\left\{\l:=\left(\l_1,\l_2,\cdots,\l_m\right)\vdash n: \text{mult}(\l_i)\le 1\ \text{if}\ \l_i\equiv 1(\text{mod}\ 2)\right\}.
\end{equation}
For $\l\in \mathcal{P}_{od}(n)$, we split $\l$ into two disjoint components as $\l:=\l_{od}\cup \l_{er}$, where 
\begin{equation}\label{oddptn2}
\l_{od}:=\left(\l_{o_1},\l_{o_2},\cdots,\l_{o_r}\right)\ \text{with}\ \l_{o_i}\equiv 1(\text{mod}\, 2),\ \l_{o_i}>\l_{o_j}\ \text{for}\ i<j,
\end{equation}
and
\begin{equation}\label{oddptn3}
\l_{er}:=\left(\l^{m_1}_{e_1}\l^{m_2}_{e_2}\cdots\l^{m_s}_{e_s}\right)\ \text{with}\ \l_{e_i}\equiv 0\ (\text{mod}\, 2),\ m_i:=\text{mult}(\lambda_{e_i}),\ \l_{e_i}>\l_{e_j}\ \text{for}\ i<j.
\end{equation}
It is easy to see that
\begin{equation}\label{oddptn4}
n=\sum_{\ell=1}^{r}\lambda_{o_\ell}+\sum_{k=1}^{s}\text{mult}(\lambda_{e_k})\cdot\lambda_{e_k}.
\end{equation}
First, we prove \eqref{aood}. For $n\in\{1,2\}$, we can verify \eqref{aood} easily. Thus, our goal is to prove \eqref{aood} for $n \ge 3$. To do so, following \eqref{oddptn2} and \eqref{oddptn3}, we decompose $\mathcal{P}_{od}(n)$ into two subsets
$\mathcal{P}_{od}(n):=\bigcup\limits_{i=1}^{2}\mathsf{P}^{[i]}_{od}(n)$,
where
\begingroup
\allowdisplaybreaks
\begin{align}\label{eqn3'}
\mathsf{P}^{[1]}_{od}(n)&:=\left\{\l\in \mathsf{P}_{od}(n):\l_{o_i}-\l_{o_{i+1}}=\l_{e_j}-\l_{e_{j+1}}=2, \l_{o_r}=1, \l_{e_s}=2, \l_{o_1}<\l_{e_1}\right\},\\\label{eqn1'}
\mathsf{P}^{[2]}_{od}(n)&:=\mathsf{P}_{od}(n)\setminus \mathcal{P}^{[1]}_{od}(n).
\end{align}
\endgroup
Thus
$\mathsf{P}^{[1]}_{od}(n)\bigcap \mathsf{P}^{[2]}_{od}(n)=\emptyset$.
For $\l=(\l_1, \l_2, \ldots\l_r) \in \mathcal{P}_{od}(n)$ with $\l_1> \l_2> \ldots> \l_r$, define a map 
$f: \mathcal{P}_{od}(n) \longrightarrow \mathcal{P}_{od}(n+1)$ by $f(\l):=(\l_1+1, \l_2, \ldots\l_r)$.
By definition of $f$, it is immediate that it is an injective map and $f(\l)\in \mathcal{P}_{od}(n+1)$.  Now, we compute $mex(\l)$ for $\l \in \mathcal{P}_{od}(n)$ and corresponding $mex(f(\l))$ for $f(\l) \in \mathcal{P}_{od}(n+1)$ individually for $\{\mathsf{P}^{[j]}_{od}(n)\}_{1 \le j\le2}$.

For $\l \in \mathsf{P}^{[2]}_{od}(n)$, we see that if $\text{mex}(\l)\equiv 1 (\text{mod 2})$, then $\text{mex}(f(\l))\equiv 1 (\text{mod 2})$. Thus using the bijection $f:\mathsf{P}^{[2]}_{od}(n) \longrightarrow \widetilde{\mathsf{P}}^{[2]}_{od}(n+1):=f(\mathsf{P}^{[2]}_{od}(n))$, we have 
\begin{equation*}
a^{o}_{od}(n)= \underset{\underset{\text{mex} \equiv 1(\text{mod}2)}{\l \in \mathcal{P}_{od}(n)}}{\sum}1\le \underset{\underset{\text{mex} \equiv 1(\text{mod}2)}{\l \in \mathsf{P}^{[1]}_{od}(n)}}{\sum}1 +\underset{\underset{\text{mex} \equiv 1(\text{mod}2)}{\widetilde{\l} \in \widetilde{\mathsf{P}}^{[2]}_{od}(n+1)}}{\sum}1.
\end{equation*}
Now, if $\l\in\mathsf{P}^{[1]}_{od}(n)$, it is of the following form 
\begin{equation}\label{prev}
\l:=\left((2m)^{\a_m},2m-1,(2m-2)^{\a_{m-1}},\cdots,2^{\a_1},1\right)\ \ \text{with}\ \a_j\in \mathbb{N}\ \ \text{for}\ \ 1\le j\le m.
\end{equation}
It is clear that $\text{mex}(\l)=2m+1\equiv 1\ (\text{mod 2})$ but $\text{mex}(f(\l))=2m+2 \equiv 0\  (\text{mod 2})$. For any such partition $\lambda \in \mathsf{P}^{[1]}_{od}(n)$, there always exists a partition, say $\mu\in \mathcal{P}_{od}(n+1)\setminus f(\mathsf{P}^{[1]}_{od}(n))$ with
$\mu=\left((2m)^{\a_m},2m-1,(2m-2)^{\a_{m-1}},\cdots,2^{\a_1+1}\right)$, where $\a_j$ are as in \eqref{prev}. Thus $\text{mex}(\mu)=1\equiv 1\ (\text{mod}\ 2)$. This infers that $a^{o}_{od}(n+1)\geq a^{o}_{od}(n)$.\\
For example, $\l=\{4,4,3,2,2,2,1\}$, so $\text{mex}(\l)=5\equiv1\ (\text{mod}2)$. Now $f(\l)=\{5,4,3,2,2,2,1\}$ and $\text{mex}(f(\l))=6\equiv0\ (\text{mod}\ 2)$. Observe that for $\l \in \mathsf{P}^{[1]}_{od}(n)$, $\mu=\{4,4,3,2,2,2,2\}\in\mathcal{P}_{od}(n+1)\setminus f(\mathsf{P}^{[1]}_{od}(n))$. It remains to show that the inequality is strict. To prove this, we consider the parity of $n+1$. For $n\equiv 0\ \left(\text{mod}\ 2\right)$, $\{2,2,\ldots2,1\} \in \mathcal{P}_{od}(n+1)\setminus f(\mathcal{P}_{od}(n))$ with corresponding $\text{mex} \equiv 1 (\text{mod }2)$. Whereas for $n\equiv 1\ \left(\text{mod}\ 2\right)$, we further split it into two cases, $\frac{n+1}{2}$ is even or odd. For $\frac{n+1}{2}$ even, $\{\frac{n+1}{2},\frac{n+1}{2}\} \in \mathcal{P}_{od}(n+1)\setminus f(\mathcal{P}_{od}(n))$ with $\text{mex} \equiv 1 (\text{mod }2)$. Finally, if $\frac{n+1}{2}$ is odd, $\{\frac{n-1}{2},\frac{n-1}{2},2\} \in \mathcal{P}_{od}(n+1)\setminus f(\mathcal{P}_{od}(n))$ with $\text{mex} \equiv 1 (\text{mod }2)$. Thus \eqref{aood} follows. 

Finally, we show \eqref{aeod}. For $n\in \{7,8,9,10\}$, \eqref{aeod} can be checked easily and so we prove \eqref{aeod} for $n\geq 11$. Following \eqref{oddptn2} and \eqref{oddptn3}, we decompose $\mathcal{P}_{od}(n)$ into two subsets
$\mathcal{P}_{od}(n):=\bigcup\limits_{i=3}^{4}\mathsf{P}^{[i]}_{od}(n)$,
where
\begingroup
\allowdisplaybreaks
\begin{align*}
\mathcal{P}^{[3]}_{od}(n)&:=\left\{\l\in \mathcal{P}_{od}(n):\l_{o_i}-\l_{o_{i+1}}=\l_{e_j}-\l_{e_{j+1}}=2, \l_{o_r}=1, \l_{e_s}=2, \l_{o_1}>\l_{e_1}\right\},\\
\mathcal{P}^{[4]}_{od}(n)&:=\mathcal{P}_{od}(n)\setminus \mathsf{P}^{[3]}_{od}(n).
\end{align*}
\endgroup
Thus
$\mathsf{P}^{[3]}_{od}(n)\bigcap \mathsf{P}^{[4]}_{od}(n)=\emptyset$.
For $\l=(\l_1, \l_2, \ldots\l_r) \in \mathcal{P}_{od}(n)$ with $\l_1> \l_2> \ldots> \l_r$, define a map 
$f: \mathcal{P}_{od}(n) \longrightarrow \mathcal{P}_{od}(n+1)$ by $f(\l):=(\l_1+1, \l_2, \ldots\l_r)$.
By definition of $f$, it is immediate that it is an injective map and $f(\l)\in \mathcal{P}_{od}(n+1)$. Next we compute $mex(\l)$ for $\l \in \mathcal{P}_{od}(n)$ and $mex(f(\l))$ for $f(\l) \in \mathcal{P}_{od}(n+1)$ for each of the subsets $\{\mathsf{P}^{[j]}_{od}(n)\}_{3 \le j\le4}$. 
We note that for $\l \in \mathsf{P}^{[4]}_{od}(n)$ with $\text{mex}(\l)\equiv 0\ (\text{mod 2})$, then $\text{mex}(f(\l))\equiv 0 (\text{mod 2})$. Thus following the bijection $f:\mathsf{P}^{[4]}_{od}(n) \longrightarrow \widetilde{\mathsf{P}}^{[4]}_{od}(n+1):=f(\mathsf{P}^{[4]}_{od}(n))$, it follows that
\begin{equation*}
a^{e}_{od}(n)= \underset{\underset{\text{mex} \equiv 0\ (\text{mod}2)}{\l \in \mathcal{P}_{od}(n)}}{\sum}1\le \underset{\underset{\text{mex} \equiv 0\ (\text{mod}2)}{\l \in \mathsf{P}^{[3]}_{od}(n)}}{\sum}1 +\underset{\underset{\text{mex} \equiv 0\ (\text{mod}2)}{\widetilde{\l} \in \widetilde{\mathsf{P}}^{[4]}_{od}(n+1)}}{\sum}1.
\end{equation*}
For $\l \in \mathsf{P}^{[3]}_{od}(n)$, it is of the following form
\begin{equation}\label{prev1}
\l:=\left(2m+1, (2m)^{\b_m}, 2m-1,(2m-2)^{\b_{m-1}},\cdots, 2^{\b_1},1\right)\ \text{with}\ \b_j\in \mathbb{N}\ \text{for}\ 1\le j\le m.
\end{equation}
Thus $\text{mex}(\lambda)=2m+2\equiv 0\ (\text{mod 2})$ but $\text{mex}(f(\lambda))=2m+1 \equiv 1\ (\text{mod 2})$ because 
$f(\lambda)=\left(2m+2, (2m)^{\b_m-1},2m-1,(2m-2)^{\b_{m-1}},\cdots,2^{\b_1},1\right)$ ($\b_j$'s are as in \eqref{prev1}). Furthermore, we observe that for $\lambda_1 \in \mathsf{P}^{[3]}_{od}(n)$, there exists $\mu\in \mathcal{P}_{od}(n+1)\setminus f(\mathsf{P}^{[3]}_{od}(n))$ which is given by 
\begingroup
\begin{equation*}
\hspace{-0.25 cm}\mu=\begin{cases}
\left((2m+2)^2,(2m)^{\b_m-1},2m-1,(2m-2)^{\b_{m-1}},\cdots, 4^{\b_2+\left\lfloor\frac{\b_1}{2}\right\rfloor},3,1\right)\ \text{if}\ \b_1\ \text{odd}\ (\ge 3),\\
\left((2m+2)^2, 2m+1,(2m)^{\b_m-1},2m-1,(2m-2)^{\b_{m-1}},\cdots,4^{\b_2+\frac{\b_1}{2}-2},3,1\right)\ \text{if}\ \b_1\ \text{even}, m\neq 2,\\
\left(6^2,5,4^{\b_2+\frac{\b_1}{2}-2},1\right)\  \text{if}\ \b_1\ \text{even}, m=2,\\
\left((2m+2)^2, (2m)^{\b_m-1},2m-1,(2m-2)^{\b_{m-1}},\cdots,4^{\b_2},3,1\right)\ \text{if}\ \b_1=1.
\end{cases}
\end{equation*}
\endgroup 
In each of the above cases, $\text{mex}(\mu)=2$ and therefore, $a^{e}_{od}(n+1)\geq a^{e}_{od}(n)$.\\
For example, consider $\l\in \mathsf{P}^{[3]}_{od}(n)$ be either of the following partitions $(7,6,6,5,4,4,4,3,2,1)$, $(7,6,6,5,4,4,3,2,2,2,2,1)$, $(7,6,5,4,4,3,2,2,2,2,2,1)$, and $(5,4,3,2,2,2,2,1)$. Then the corresponding $\mu\in\mathcal{P}_{od}(n+1)\setminus f(\mathcal{P}_{od}(n))$ are given by $(8,8,6,5,4,4,4,3,1),(8,8,7,6,4,4,4,3,1)$, $(8,8,5,4,4,4,4,3,1), (6,6,5,4,1)$ respectively. Finally we show the inequality is strict. To show it, again we consider parity of $n$. For $\frac n2$ even, $(\frac n2, \frac n2, 1) \in \mathcal{P}_{od}(n+1)\setminus f(\mathcal{P}_{od}(n))$ with $\text{mex} \equiv 0\ (\text{mod}\ 2)$. Note that for $n \geq 14$ with $n$ even, choose $\overline{\lambda}:=(4^{\g_2},3,2^{\g_1},1)\in\mathcal{P}_{od}(n)$ with $\g_2\ge 2, \g_1\ge 1$, and $\text{mex} \equiv 1\ (\text{mod }\ 2)$. Consequently, $f(\overline{\lambda})=(5,4^{\g_2-1},3,2^{\g_1},1) \in f(\mathcal{P}_{od}(n))\subset \mathcal{P}_{od}(n+1)$ with $\text{mex} \equiv 0\ (\text{mod }\ 2)$. Finally, we consider $\frac{n+1}{2}$ is even and in this case for $n\ge 11$, $(\frac{n+1}{2},\frac{n-1}{2},1) \in \mathcal{P}_{od}(n+1)\setminus f(\mathcal{P}_{od}(n))$ with $\text{mex} \equiv 0\ (\text{mod }\ 2)$. For $n \geq 14$ with $n$ odd, there exists a partition $(3,2^{\d},1) \in \mathcal{P}_{od}(n+1)\setminus f(\mathcal{P}_{od}(n))$ with $\d\ge 5$ and $\text{mex} \equiv 0\ (\text{mod}\ 2)$. This concludes the proof of \eqref{aeod}.\qed  

\emph{Proof of Theorem \ref{asymp2}}: From Theorem \ref{thm_aod}, we have  
\begin{equation*}
\sum_{n=0}^{\infty}a^{o}_{od}(n)q^n= \frac{(-q;q^2)_{\infty}}{(q^2;q^2)_{\infty}}\sum_{n=0}^{\infty}\frac{q^{2n^2+n}}{(-q;q^2)_{n+1}}=:A_1(q) B_1(q),
\end{equation*} where \begin{equation}\label{A1}
A_1(q)=\frac{(-q;q^2)_{\infty}}{(q^2;q^2)_{\infty}}=\frac{\left(q^2;q^2\right)_{\infty}}{\left(q;q\right)_{\infty}\left(q^4;q^4\right)_{\infty}} \ \ \text{and} \ \ B_1(q)=\sum_{n=0}^{\infty}\frac{q^{2n^2+n}}{(-q;q^2)_{n+1}}.
\end{equation}
Using \eqref{eqn1}, we have
\begin{equation}\label{A1(q)asymp}
A_1(e^{-t})\sim \sqrt{\frac{t}{\pi}}e^{\frac{\pi^2}{8t}}\ \ \text{as} \ \ t \rightarrow 0^{+}.
\end{equation}
Next, as $t\to 0^{+}$, we get
\begin{align*}
B_1(e^{-t})=\sum_{n=0}^{\infty}\frac{e^{-(1+2+3+\ldots+2n) t}}{(-e^{-t};e^{-2t})_{n+1}}\to \sum_{n=0}^{\infty}\frac{1}{2^{n+1}}=1. 
\end{align*}
Consequently with \eqref{A1(q)asymp}, we have
$$A_1(e^{-t})B_1(e^{-t})\sim \sqrt{\frac{t}{\pi}}e^{\frac{\pi^2}{8t}}\ \ \text{as} \ \ t \rightarrow 0^{+}.$$
Using this, Proposition \ref{increasing2}, and applying Proposition \ref{ingham} with $\alpha=\frac{1}{\sqrt{\pi}}$, $\beta=\frac 12$ and $C=\frac{\pi^2}{8}$, we have 
\begin{equation*}
a^{o}_{od}(n) \sim \frac{e^{\pi\sqrt{\frac{n}{2}}}}{4\sqrt{2}n}.
\end{equation*} 
This concludes the proof of the first claim of \eqref{asympaodoe}. 

Next we prove the second claim in \eqref{asympaodoe}. By Theorem \ref{thm_aod}, we have  \begin{equation*}
\sum_{n=0}^{\infty}a^{e}_{od}(n)q^n= \frac{(-q;q^2)_{\infty}}{(q^2;q^2)_{\infty}}\sum_{n=0}^{\infty}\frac{q^{2n^2-n}(1-q^{2n})}{(-q;q^2)_{n}}=:A_1(q) B_2(q),
\end{equation*} where $A_1(q)$ be as in \eqref{A1} and 
 $B_2(q)=\sum_{n=0}^{\infty}\frac{q^{2n^2-n}(1-q^{2n})}{(-q;q^2)_{n}}$.
We note that
\begin{align*}
B_2(e^{-t})&=\sum_{n=1}^{\infty}\frac{e^{-(1+2+3+\ldots+(2n-1)) t}(1-e^{-2nt})}{(-e^{-t};e^{-2t})_{n}}  \nonumber\\ 
&=
\left(\frac{1-e^{-2t}}{1+e^{t}}+\frac{e^{-2t}(1-e^{-4t})}{(1+e^{t})(1+e^{3t})}+\cdots+\frac{e^{-(2+4+\ldots+(2n-2))t}(1-e^{-2nt})}{(1+e^{t})(1+e^{3t})\cdots(1+e^{(2n-1)t})}+\cdots\right). 
\end{align*}
Since, $B_2(e^{-t})\rightarrow 0$ as $t \rightarrow 0^{+}$ and
$\lim_{t\rightarrow0^{+}} \frac{\frac{e^{-(2+4+\ldots+(2n-2))t}(1-e^{-2nt})}{(1+e^{t})(1+e^{3t})\cdots (1+e^{(2n-1)t})}}{t}=\frac{n}{2^{n-1}}$,
 we have $B_2(e^{-t})\sim \sum_{n=1}^{\infty}\frac{n}{2^{n-1}}=4t$. Combining this with \eqref{A1(q)asymp}, we get $A_2(e^{-t})B_2(e^{-t})\sim \frac{4t^{\frac 32}}{\sqrt{\pi}}e^{\frac{\pi^2}{8t}}$ as $t \rightarrow 0^{+}.$ Consequently, by Proposition \ref{increasing2} and Proposition \ref{ingham} with $\alpha=\frac{4}{\sqrt{\pi}}$, $\beta=\frac 32$ and $C=\frac{\pi^2}{8}$, we have 
$a^{e}_{od}(n) \sim \frac{e^{\frac{\pi\sqrt{n}}{2}}}{\sqrt{2}n}$.
This concludes the proof.\qed

\section{Proof of Theorems \ref{thm_sigma}--\ref{thm_meex-4e}}\label{sec3}

In this section, we first prove Theorem \ref{thm_sigma} in two different ways, via $q$-series transformations and combinatorial framework. Next, we present two proofs of Theorem \ref{rem2} by using Theorem \ref{thm_sigma} and constructing a bijection (with an explicit examples) in spirit of the work by Ballantine and Merca \cite{BM}. Then we move on to present proofs of Theorems \ref{increasing} and \ref{asymp}. Finally, we conclude this section by proving Theorem \ref{thm_meex-4e}.

\emph{Proof of Theorem \ref{thm_sigma}}: Using \eqref{def3}, we have
\begin{eqnarray}
&&\sum_{n=0}^{\infty}\sigma_{od}\text{moex}(n)q^n = \sum_{n=0}^{\infty}\frac{(2n+1) q^{1+3+5+\cdots+(2n-1)}(-q^{2n+3};q)_{\infty}}{(q^2;q^2)_{\infty}}= \frac{(-q;q^2)_{\infty}}{(q^2;q^2)_{\infty}} \sum_{n=0}^{\infty}\frac{(2n+1)q^{n^2}}{(-q;q^2)_{n+1}} \nonumber \\ 
&&= \frac{(-q;q^2)_{\infty}}{(q^2;q^2)_{\infty}}\!\sum_{n=0}^{\infty}\frac{(2n+1)q^{n^2}}{(-q;q^2)_{n}} \Biggl(1\!-\!\frac{q^{2n+1}}{1+q^{2n+1}}\Biggr)\!=\!\frac{(-q;q^2)_{\infty}}{(q^2;q^2)_{\infty}} \Biggl(\sum_{n=0}^{\infty}\frac{(2n+1)q^{n^2}}{(-q;q^2)_{n}}\!-\! \sum_{n=0}^{\infty}\frac{(2n+1)q^{(n+1)^2}}{(-q;q^2)_{n+1}} \Biggr) \nonumber \\ 
&&= \frac{(-q;q^2)_{\infty}}{(q^2;q^2)_{\infty}} \Biggl(\sum_{n=0}^{\infty}\frac{(2n+1)q^{n^2}}{(-q;q^2)_{n}} - \sum_{n=1}^{\infty}\frac{(2n-1)q^{n^2}}{(-q;q^2)_{n}} \Biggr)= \frac{(-q;q^2)_{\infty}}{(q^2;q^2)_{\infty}}\Biggl(1+2 \cdot \sum_{n=1}^{\infty}\frac{q^{n^2}}{(-q;q^2)_{n}} \Biggr)\nonumber\\
&&=\frac{(-q;q^2)_{\infty}}{(q^2;q^2)_{\infty}}\Bigl(1+\s^*(-q) \Bigr),\nonumber
\end{eqnarray}
using \eqref{cohensigma} (with $q\mapsto -q$) in the final step. This concludes the proof of \eqref{thm_sigma1}. 

Analogously, following \eqref{def3}, we have
\begin{eqnarray}\nonumber 
&&\sum_{n=0}^{\infty}\sigma_{od}\text{meex}(n)q^n = (-q;q^2)_{\infty}\sum_{n=0}^{\infty}\frac{2n q^{2+4+6+\cdots+(2n-2)}}{\displaystyle \prod_{\underset{m \neq n}{m=1}}^{\infty}(1-q^{2m})}= \frac{(-q;q^2)_{\infty}}{(q^2;q^2)_{\infty}}\sum_{n=0}^{\infty} 2n q^{2 \binom{n}{2}}(1-q^{2n})\\\nonumber
&&= \frac{(-q;q^2)_{\infty}}{(q^2;q^2)_{\infty}}\Biggl(\sum_{n=0}^{\infty} 2n q^{2 \binom{n}{2}}-\sum_{n=0}^{\infty} 2n q^{2\binom{n+1}{2}} \Biggr)=\frac{(-q;q^2)_{\infty}}{(q^2;q^2)_{\infty}}\Biggl(\sum_{n=0}^{\infty} 2n q^{2 \binom{n}{2}}-\sum_{n=1}^{\infty} (2n-2)q^{2\binom{n}{2}} \Biggr) \\\nonumber
&& = 2\frac{(-q;q^2)_{\infty}}{(q^2;q^2)_{\infty}}\sum_{n=0}^{\infty}   q^{2\binom{n+1}{2}}=  2\frac{(-q;q^2)_{\infty}}{(q^2;q^2)_{\infty}} \frac{(-q^2;q^2)_{\infty}(q^2;q^2)_{\infty}}{(q;q^2)_{\infty}(-q;q^2)_{\infty}}=2 \sum_{n=0}^{\infty}p_{ed}(n)q^n,
\end{eqnarray}
using \eqref{Gauss} (with $q\mapsto q^2$) in the penultimate step. This concludes the proof.\qed 

\emph{An alternative proof of Theorem \ref{thm_sigma}}: First we prove \eqref{thm_sigma1}. Let $M^{od}_{m}(n)$ denote the number of partitions $\lambda\vdash n$ into distinct odd parts with $moex(\lambda)>m$. For $j\in \mathbb{N}$, we first show that 
\begin{equation}\label{Mod}
M^{od}_{2j-1}(n)=p_{od}(n-j^2,2j-1)\ \ \text{and}\ \ M^{od}_{0}(n)=p_{od}(n),
\end{equation}
where $p_{od}(n,2j-1)$ denotes the number of partitions of $n$ into distinct odd parts with smallest odd part $> 2j-1$. Let $\lambda\vdash n$ into distinct odd parts with $moex(\lambda)>2i-1$. Then $\lambda$ does not contain $1,3,\ldots,2j-1$ as parts, i.e., $\lambda\vdash n-j^2$ into distinct odd parts with smallest odd part $>2j-1$. Similarly the converse is also true and this proves \eqref{Mod}. Next we show that 
\begin{equation}\label{Mod2}
\sigma_{od}moex(n)=M^{od}_0(n)+2\sum_{j\ge 1}M^{od}_{2j-1}(n).
\end{equation}
Since $M^{od}_{2j-1}(n)=M^{od}_{2j}(n)$, each $\lambda\vdash n$ with $moex(\lambda)=2j+1$ is counted exactly twice for each of $M^{od}_{1}(n),M^{od}_{3}(n),\ldots,M^{od}_{2j-1}(n)$ and once for $M^{od}_{0}(n)$. Thus counted in total $2j+1$ times (each partition with  $moex(\lambda)=1$ is counted exactly once in $M^{od}_{0}(n)$) and consequently, the resulting sum on the right-hand side of \eqref{Mod2} yields $\sigma_{od}\text{moex}(n)$. This proves \eqref{Mod2}. Since $\sum_{n\ge 0}p_{od}(n-j^2,2j-1)q^n=\frac{(-q^{2j+1};q^2)_{\infty}}{(q^2;q^2)_{\infty}}$, which counts the number of partitions of $\l\vdash n-j^2$ with smallest odd part $>2j-1$ is the coefficient of $q^{n-j^2}$ in $\frac{(-q^{2j+1};q^2)_{\infty}}{(q^2;q^2)_{\infty}}$, i.e., the coefficient of $q^n$ in $q^{j^2}\frac{(-q^{2j+1};q^2)_{\infty}}{(q^2;q^2)_{\infty}}$. Thus, by \eqref{Mod} and \eqref{Mod2}, we have
\begin{align*}
&\sum_{n=0}^{\infty}\sigma_{od}moex(n)q^{n}
= \sum_{n=0}^{\infty}p_{od}(n)q^n+2\sum_{n=0}^{\infty}q^n\sum_{j\ge 1}p_{od}(n-j^2,2j-1)\\ 
&= \frac{(-q;q^2)_{\infty}}{(q^2;q^2)_{\infty}}+\frac{2}{(q^2;q^2)_{\infty}}\sum_{j\ge 1} q^{j^2}(-q^{2i+1};q^2)_{\infty}= \frac{(-q;q^2)_{\infty}}{(q^2;q^2)_{\infty}}+2 \frac{(-q;q^2)_{\infty}}{(q^2;q^2)_{\infty}} \sum_{j\ge 1}\frac{q^{j^2}}{(-q;q^2)_{j}} \\
&= \frac{(-q;q^2)_{\infty}}{(q^2;q^2)_{\infty}}(1+\sigma^{*}(-q)).
\end{align*} 
This completes the proof of \eqref{thm_sigma1}. 

Next we prove \eqref{thm_sigma2}. Let $\overline{M}^{od}_{m}(n)$ denote the number of partitions $\lambda\vdash n$ into distinct odd parts with $meex(\lambda)>m$. For $j\in \mathbb{N}_0$, we show that 
\begin{equation}\label{Mod3}
\overline{M}^{od}_{2j}(n)=p_{od}\left(n-(j^2+j),2j\right)=p_{od}(n-j(j+1)).
\end{equation}
Let $\lambda\vdash n$ be a partition into distinct odd parts with $meex(\lambda)>2j$. Precisely $\lambda\vdash n$ is of the form $\lambda=(\lambda_{o})\cup(\l^{m_k}_{e_k},\l^{m_{k-1}}_{e_{k-1}},\cdots,\l^{m_1}_{e_1})$, where $m_{\ell}=\text{mult}\left(\l_{e_{\ell}}\right)$, $\lambda_{e_s}<\lambda_{e_{t}}$ for $s<t$, and $\lambda_{e_k}> 2j$. Then each such $\lambda$ is counted by the existence of $2,4,\cdots,2j$ as parts with remaining part as an arbitrary partition of $n-(2+4+\cdots+2j)=n-j(j+1)$ into distinct odd parts and this proves \eqref{Mod3}. Next we claim that 
\begin{equation}\label{Mod4}
\sigma_{od}\text{meex}(n)=2\sum_{j\ge 0}\overline{M}^{od}_{2j}(n).
\end{equation}
Since $\overline{M}^{od}_{2j}(n)=\overline{M}^{od}_{2j+1}(n)$, each partition $\lambda\vdash n$ with $meex(\lambda)=2j+2$ is counted exactly twice in each of $\overline{M}^{od}_{0}(n),\overline{M}^{od}_{4}(n),\cdots,\overline{M}^{od}_{2i}(n)$, i.e., counted in total $2j+2$ times and consequently their resulting sum gives the right hand side of \eqref{Mod4} which proves \eqref{Mod4}. Thus, using \eqref{Mod3} and \eqref{Mod4}, we have
\begin{align*}
\sum_{n=0}^{\infty}\sigma_{od}\text{meex}(n)q^{n} &= 2\sum_{n=0}^{\infty}q^{n}\sum_{j\ge 0}p_{od}(n-j(j+1))= 2\frac{(-q;q^2)_{\infty}}{(q^2;q^2)_{\infty}}\sum_{j\ge 1} q^{j(j+1)} \\&=2 \frac{(-q;q^2)_{\infty}}{(q^2;q^2)_{\infty}} \frac{(-q^2;q^2)_{\infty}(q^2;q^2)_{\infty}}{(q;q^2)_{\infty}(-q;q^2)_{\infty}}= 2 \sum_{n=0}^{\infty}p_{ed}(n)q^n, 
\end{align*} 
using \eqref{Gauss} in the penultimate step. This concludes the proof of \eqref{thm_sigma2}.\qed

\emph{Proof of Theorem \ref{rem2}}: Comparing coefficients of $q^n$ in \eqref{thm_sigma2}, we conclude the proof of Theorem \ref{rem2}.\qed

\emph{A combinatorial proof of Theorem \ref{rem2}}: To construct the bijection, we interpret partitions into distinct even parts in the setting of two-colored partitions. Let $D^{e}_2(n)$ be the set of two-colored\footnote{We label two colors using $0$ and $1$.} partitions of $n$ into distinct parts where only even parts can only appear with two colors. Denote $d^{e}_2(n)=|D^{e}_2(n)|$. For example, $4_0+4_1+3+1 \in D^{e}_2(12)$. Due to Euler, we know that $(-q)_{\infty}=\frac{1}{(q;q^2)_{\infty}}$. Consequently, we see that
\begin{equation*}
\sum_{n=0}^{\infty}p_{ed}(n)q^n=\frac{(-q^2;q^2)_{\infty}}{(q;q^2)_{\infty}}=(-q^2;q^2)_{\infty}(-q)_{\infty}=(-q^2;q^2)^2_{\infty}(-q;q^2)_{\infty}. 
\end{equation*}
Thus from the definition of $D^{e}_2(n)$, we see that $p_{ed}(n)=d^{e}_2(n)$. Now to present a combinatorial proof of \eqref{thm_sigma2}, from \eqref{Mod3} and \eqref{Mod4}, it is now enough to construct a bijection between $\underset{k\geq0}{\bigcup}P_{od}(n-k(k+1))$ and $D^{e}_2(n)$, where $P_{od}(n-k(k+1))$ is the set of all partitions of $n-(2+4+\cdots+(2k-2)+2k)$ into restricted odd parts. For $k\in \mathbb{N}$, let $\d(k)$ be the stairecse partition $\d(k)=k+(k-1)+\cdots+3+2+1$ and $\d(0):=1$. Following \cite[Definition 1]{BM}, here we have an example of a stairecase profile of a Ferrers diagram, shown below:
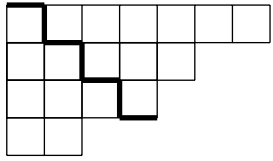
\begin{center}
	\begin{tikzpicture}[scale=0.5, line width=0.5pt]
	\draw (0,0) grid (7,1);
	\draw[line width=2pt] (0,1) -- ++(1,0);
	\draw (0,0) grid (5,-1);
	\draw[line width=2pt] (1,1) -- ++(0,-1);
	\draw[line width=2pt] (1,0) -- ++(1,0);
	\draw[line width=2pt] (2,0) -- ++(0,-1);
	\draw[line width=2pt] (2,-1) -- ++(1,0);
	\draw[line width=2pt] (3,-1) -- ++(0,-1);
	\draw[line width=2pt] (2,-1) -- ++(1,0);
	\draw[line width=2pt] (3,-2) -- ++(1,0);
	\draw (0,0) grid (4,-2);
	\draw (0,0) grid (2,-3);
	\end{tikzpicture}
\end{center}
Given a Ferrers diagram of a partition $\lambda$ into
distinct parts, the shifted Ferrers diagram of $\lambda$ is the diagram in which row $j$ is
shifted $j-1$ units to the right. Define a map $\phi: \underset{k\geq0}{\bigcup}P_{od}(n-k(k+1)) \longrightarrow D^{e}_2(n)$ as follows. Note that a partition $\lambda$ of $n$ into only even parts corresponds to a partition of $\frac{n}{2}$ into an ordinary partition by dividing all parts of $\lambda$ by 2. Let $\lambda:=\lambda_{od} \sqcup \lambda_{er}\in P_{od}(n-k(k+1))$ for some $k \geq 0$, where $\l_{od}:=(\l_{o_1},\l_{o_2},\cdots,\l_{o_r})$ with $\l_{o_{\ell}}\equiv 1(\text{mod}\ 2)$, $\l_{o_s}>\l_{o_t}$ for $s<t$, 
and $\l_{er}:=(\l^{m_1}_{e_1},\l^{m_2}_{e_2}\cdots,\l^{m_s}_{e_s})$ with $\l_{e_{\ell}}\equiv 0\ (\text{mod}\, 2)$, $m_{\ell}:=\text{mult}(\lambda_{e_{\ell}})$ with $\l_{e_s}>\l_{e_t}$ for $s<t$.
Define  
$n_1:=n-k(k+1)=\sum_{\ell=1}^{r}\lambda_{o_\ell}+\sum_{k=1}^{s}\text{mult}(\lambda_{e_k})\cdot\lambda_{e_k}$.
Now consider the partition $\l^{'}:=((\frac{\l_{e_1}}{2})^{m_1},(\frac{\l_{e_2}}{2})^{m_2},\cdots,(\frac{\l_{e_s}}{2})^{m_s})\vdash \frac{n_1-\sum_{\ell=1}^{r}\lambda_{o_\ell}}{2}$. Append a diagram with rows of lengths $1,2,\cdots,k$, i.e., the Ferrers diagram of $\d(k)$ rotated by $180^\circ$, to
the top of Ferrers diagram of the conjugate partition $\lambda^{'}$. Here are the following steps of constructions.
\begin{enumerate}
\item First, we draw the
staircase profile of the new diagram. Let $\l^{'}_0$ be the partition whose parts are the length of the columns to the left of the staircase profile and let $\l^{'}_1$ be the partition whose parts are the length of the rows to the right of the staircase profile.  Then
$\l^{'}_0$ and $\l^{'}_1$ are partitions with distinct parts. Since the length difference of $\l^{'}_0$ and $\l^{'}_1$, i.e., $\ell(\l^{'}_0)-\ell(\l^{'}_1)$ is $k$ when the staircase profile ends with a horizontal line and is equal to $k+1$ when staircase profile ends with a vertical line respectively, so that we have $k\leq \ell(\l^{'}_0)-\ell(\l^{'}_1) \leq k+1$.
\item Then we multiply each parts of $\l^{'}_0$ and $\l^{'}_1$ by $2$.
\item  Finally, we color the parts of $\l^{'}_0$ with color $k\ (\text{mod}\ 2)$ and the parts of $\l^{'}_1$ with color $k+1\ (\text{mod}\ 2)$ so that the new partitions formed, say $\l^{"}_0$ and $\l^{"}_1$ respectively, are partitions with only even distinct parts and therefore $\l^{"}_0 \sqcup \l^{"}_1$ forms a partition of $n-\sum_{\ell=1}^{r}\lambda_{o_\ell}$ into partitions with only even distinct two color parts. Thus $\phi(\l):=\l^{"}_0 \sqcup \l^{"}_1\sqcup\l_{od}$ is a partition in $D^{e}_{2}(n)$.
\end{enumerate}
From this construction, through the staircase profile of diagram, it is clear that $\phi$ is a well-defined map. We see that for $k$ even, $\ell(\l^{"}_0)\geq \ell(\l^{"}_1)$ and for $k$ odd we have $\ell(\l^{"}_1)> \ell(\l^{"}_0)$. Thus from the range of the corresponding length difference of $\ell(\l^{"}_0)$ and $\ell(\l^{"}_1)$ and taking into consideration the way of coloring, for $k\in \mathbb{N}_{\geq 0}$,  $\phi\left(P_{od}\left(n-k(k+1)\right)\right)\cap\phi\left(P_{od}\left(n-(k+1)(k+2)\right)\right)=\emptyset$. Also by the construction of $\phi$, all elements of $\phi\left(P_{od}\left(n-k(k+1)\right)\right)$ are distinct. Thus, $\phi$ is injective. Conversely, we split $\overline{\lambda}:=\overline{\lambda}_{od} \sqcup \overline{\lambda^0}_{er}\sqcup\overline{\l^1}_{er}\in D^{e}_2(n)$, where $\overline{\l}_{od}:=(\overline{\l}_{o_1},\overline{\l}_{o_2},\cdots,\overline{\l}_{o_r})$ with $\overline{\l}_{o_{\ell}}\equiv 1\ (\text{mod}\ 2)$, $\overline{\l}_{o_s}>\overline{\l}_{o_t}$ for $s<t$,  
and $\overline{\l^{j}}_{er}:=(\overline{\l^{j}}_{e_1},\overline{\l^{j}}_{e_2},\cdots,\overline{\l^{j}}_{e_{s_j}})$ with $\overline{\l^{j}}_{e_{\ell}}\equiv 0\ (\text{mod}\ 2)$, $\overline{\l^{j}}_{e_s}>\overline{\l^{j}}_{e_{t}}$ for $s<t$ and $j\in\{0,1\}$.
We set 
$n=\sum_{\ell=1}^{r}\overline{\lambda}_{o_\ell}+\sum_{k=1}^{s_0}\overline{\lambda^{0}}_{e_k}+\sum_{k=1}^{s_1}\overline{\lambda^{1}}_{e_k}$,
and consider the ordinary partitions $\overline{\l}_j:=(\frac{\overline{\l^j}_{e_1}}{2},\frac{\overline{\l^j}_{e_2}}{2},\cdots,\frac{\overline{\l^j}_{e_{s_j}}}{2})$ with $j\in\{0,1\}$. Next we define the map $\phi_1: D^{e}_2(n)\longrightarrow \underset{k\geq 0}{\bigcup} P_{od}\left(n-k(k+1)\right)$ as follows. For $\ell(\overline{\l}_0)\geq \ell(\overline{\l}_1)$, let $k':=\ell(\overline{\l}_0) -\ell(\overline{\l}_1)$. Let $k:=k'+\frac{(-1)^{k'}-1}{2}$ so that if $k'$ is even, $k=k'$ and otherwise $k=k'-1$. In this setting, our construction is done via the following steps: 
\begin{enumerate}
	\item First, we remove the top $k$ rows, i.e., the rotated Ferrers diagram of $\d(k)$ from the conjugate of the shifted diagram of $\overline{\l}_0$.
	\item Then we join the remaining diagram with the shifted digram of $\overline{\l}_1$ so that they align at the top.
	\item In the final step, we multiply each part by $2$ and denote the obtained partition by $\overline{\l}_2$. Then the partition $\phi_1(\overline{\l}):=\overline{\l}_{od}\sqcup\overline{\l}_2\in P_{od}\left(n-k(k+1)\right)$. 
\end{enumerate}
For $\ell(\overline{\l}_1)> \ell(\overline{\l}_0)$, we proceed as above with $k:=k'-\frac{(-1)^{k^{'}}+1}{2}$ where $k':=\ell(\overline{\l}_1) -\ell(\overline{\l}_0)$ with interchanging the role of $\overline{\l}_0$, $\overline{\l}_1$ by $\overline{\l}_1$, $\overline{\l}_0$ respectively. Then we have that the resulting partition $\phi_1(\overline{\l})\in P_{od}\left(n-k(k+1)\right)$. Thus, from the injectivity of $\phi$, it is clear that $\phi_1$ is well defined and consequently, it follows that $\phi_1$ is injective. Thus $\phi\circ\phi_1=\text{Id}_{D^{e}_2(n)}$ and $\phi_1\circ\phi=\text{Id}_{\underset{k\geq0}{\bigcup}P_{od}\left(n-k(k+1)\right)}$ imply that $\phi$ is bijective.\qed 

Let us illustrate the bijection with the following example. Let $n=16$, $k=1$ and take $\l=(6,4,3,1)\in P_{od}(14)$. Therefore, the corresponding $\l^{'}=(3,2)$ as $\l_{od}=(3,1)$. Now on the top of $\l^{'}$ we add one box, i.e., rotated Ferrers diagram of $\d(1)$ (since, $k=1$) and draw staircase profile. Then, we have $\l^{'}_1=(3^1,2^1)$, $\l^{'}_0=(1^0)$ so that the corresponding partition $\phi(\l)=(6^1,4^1,3,2^0,1) \in D^{e}_2(16)$ (after multiplication of each part of $\l^{'}_1$ and $\l^{'}_0$ by 2).
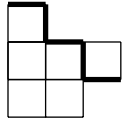
\begin{center}
	\begin{tikzpicture}[scale=0.5, line width=0.5pt]
	\draw (0,0) grid (1,1);
	\draw[line width=2pt] (0,1) -- ++(1,0);
	\draw (0,0) grid (3,-1);
	\draw[line width=2pt] (1,1) -- ++(0,-1);
	\draw[line width=2pt] (1,0) -- ++(1,0);
	\draw[line width=2pt] (2,0) -- ++(0,-1);
	\draw[line width=2pt] (2,-1) -- ++(1,0);
	\draw (0,0) grid (2,-2);
	\end{tikzpicture}
\end{center}

Conversely, for $\overline{\l}=(6^1,4^1,3,2^0,1) \in D^{e}_2(16)$, here, $k^{'}=\ell(\overline{\l}_1)-\ell(\overline{\l}_0)=1$ and therefore $k=1$ as $\ell(\overline{\l}_1)>\ell(\overline{\l}_0)$. Now the corresponding $\overline{\l}_0=(1^0)$ and $\overline{\l}_1=(3^1,2^1)$.  The diagrams of the conjugate of the shifted diagram of $\overline{\l}_0$ and $\overline{\l}_1$ are respectively as follows. 
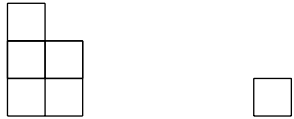
\begin{center}
	\begin{tikzpicture}[scale=0.5, line width=0.5pt]
	\draw (0,0) grid (1,1);
	\draw (0,0) grid (2,-1);
	\draw (0,0) grid (2,-2);
	\end{tikzpicture}
	\hspace*{2cm}
	\begin{tikzpicture}[scale=0.5, line width=0.5pt]
	\draw (0,0) grid (1,1);
	\end{tikzpicture}
\end{center}
Since $k=1$, we remove first row from the the top of the conjugate of the shifted diagram of $\overline{\l}_1$  
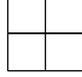
\begin{center}
	\begin{tikzpicture}[scale=0.5, line width=0.5pt]
	\draw (0,0) grid (2,1);
	\draw (0,0) grid (2,-1);
	\end{tikzpicture}
\end{center} and  join the remaining diagram and the shifted digram of $\overline{\l}_0$ so they align
at the top. 
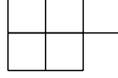
\begin{center}
	\begin{tikzpicture}[scale=0.5, line width=0.5pt]
	\draw (0,0) grid (3,1);
	\draw (0,0) grid (2,-1);
	\end{tikzpicture}
\end{center}
After multiplication of each part of the resulting partition $(3,2)$ by $2$, we get $\phi^{-1}(\overline{\l})=(6,4,3,1)\in P_{od}(14)$. 

\emph{Proof of Theorem \ref{increasing}}: Let $\mathcal{P}_{od}(n)$ be as in \eqref{oddptn1}. For $\l\in \mathcal{P}_{od}(n)$, we split $\l$ as given in \eqref{oddptn2} and \eqref{oddptn3}. From \eqref{oddptn4}, we already know that $n=\sum_{\ell=1}^{r}\lambda_{o_\ell}+\sum_{k=1}^{s}\text{mult}(\lambda_{e_k})\cdot\lambda_{e_k}$.
First, we decompose $\mathcal{P}_{od}(n)$ into five subsets
\begin{equation}\label{decomp}
\mathcal{P}_{od}(n):=\bigcup\limits_{j=1}^{5}\mathcal{P}^{[j]}_{od}(n),
\end{equation}
where
\begingroup
\allowdisplaybreaks
\begin{align}\label{eqn3}
\mathcal{P}^{[1]}_{od}(n)&:=\left\{\l\in \mathcal{P}_{od}(n):\l_{o_j}-\l_{o_{j+1}}=2\ \text{for}\ 1\le j\le r-1,\ \l_{o_r}=1,\ \text{and}\ \l_{o_1}>\l_{e_1}\right\},\\\label{eqn4}
\mathcal{P}^{[2]}_{od}(n)&:=\left\{\l\in \mathcal{P}_{od}(n):\l_{o_j}-\l_{o_{j+1}}=2\ \text{for}\ 1\le j\le r-1, \l_{o_r}>1\right\},\\\label{eqn5}
\mathcal{P}^{[3]}_{od}(n)&:=\left\{\l\in \mathcal{P}_{od}(n):\l_{o_j}-\l_{o_{j+1}}=2\ \text{for}\ 1\le j\le r-1,\ \l_{o_r}=1,\ \text{and}\ \l_{o_1}<\l_{e_1}\right\},\\\label{eqn6}
\mathcal{P}^{[4]}_{od}(n)&:=\left\{\l\in \mathcal{P}_{od}(n):\l_{o_j}-\l_{o_{j+1}}>2\ \text{for some}\ 1\le j\le r-1,\ \text{and}\ \l_{o_1}>\l_{e_1}\right\},\\\label{eqn7}
\mathcal{P}^{[5]}_{od}(n)&:=\left\{\l\in \mathcal{P}_{od}(n):\l_{o_j}-\l_{o_{j+1}}>2\ \text{for some}\ 1\le j\le r-1,\ \text{and}\ \l_{o_1}<\l_{e_1}\right\}.
\end{align}
\endgroup
Thus
\begin{equation}\label{disjoint}
\mathcal{P}^{[k]}_{od}(n)\bigcap \mathcal{P}^{[\ell]}_{od}(n)=\emptyset\ \text{for}\ 1\le k\neq \ell\le 5.
\end{equation}
Define a map 
\begin{eqnarray}\label{map}
& &	f: \mathcal{P}_{od}(n) \longrightarrow \mathcal{P}_{od}(n+1) \nonumber \\ & & \hspace*{-0.5cm} \l=\l_{od}\cup\l_{er} \mapsto \begin{cases}
(\lambda_{o_1}+1, \lambda_{o_2}, \cdots \lambda_{o_r})\cup\l_{er}, \quad \text{if}\ \l\in  \mathcal{P}^{[k]}_{od}(n)\ \text{with}\ k\in\{1,2,4\},\\
\l_{od}\cup	(\lambda_{e_1}+1, \lambda_{e_2}, \cdots \lambda_{e_s}), \quad \text{if}\ \l\in  \mathcal{P}^{[k]}_{od}(n)\ \text{with}\ k\in\{3,5\}.
\end{cases}
\end{eqnarray}
By definition of $f$, it is clear that $f$ is an injective map and $f(\l)\in \mathcal{P}_{od}(n+1)$.  Next, we compute the $moex$ of $\lambda\in \mathcal{P}_{od}(n)$ and that of for $f(\lambda)\in \mathcal{P}_{od}(n+1)$ individually for $\{\mathcal{P}^{[j]}_{od}(n)\}_{1\le j\le 5}$ defined in \eqref{eqn3}--\eqref{eqn7}.  For $\l\in \mathcal{P}^{[2]}_{od}(n)$, by definition of $moex$ of a partition, \eqref{eqn4}, and \eqref{map}, we have
\begin{equation}\label{eqn8}
moex(\l)=1=moex(f(\l)).
\end{equation}
For $\l\in \mathcal{P}^{[3]}_{od}(n)$, we see that $moex(\l)=\l_{o_1}+2$ and by \eqref{eqn5} and \eqref{map}, we have
\begin{equation*}
moex(f(\l))=\begin{cases}
\l_{o_1}+2,\quad \text{if}\ \l_{e_1}+1>\l_{o_1}+2,\\
\l_{o_1}+4,\quad \text{if}\ \l_{e_1}+1=\l_{o_1}+2.
\end{cases}
\end{equation*}
Therefore, if $\l\in \mathcal{P}^{[3]}_{od}(n)$, it follows that
$moex(\l)\le moex(f(\l))$.
For example, $\lambda=(4,3,2,2,1)$ is a partition of $12$ with $moex(\lambda)=5$ and $moex(f(\lambda))=7$ as $\lambda_{e_1}+1=4+1=3+2=\lambda_{o_1}+2$. Also $\lambda=(8,5,3,2,2,1)$ is a partition of $21$ with $moex(\lambda)=7$ and $moex(f(\lambda))=7$ since $\lambda_{e_1}+1=8+1>5+2=\lambda_{o_1}+2$. For $\l\in \mathcal{P}^{[k]}_{od}(n)$ with $k\in \{4,5\}$, by definitions \eqref{eqn6} and \eqref{eqn7}, we set
$j_0:=\underset{1\le j\le r-1}{\min}\{\l_{o_j}-\l_{o_{j+1}}>2\},$
and write $\l_{od}$ as
$\l_{od}=(\l_{o_1},\l_{o_2},\cdots,\l_{o_{j_0}},\l_{o_{j_0+1}},\l_{o_{j_0+1}}-2,\cdots,3,1)$. 
Consequently, by \eqref{map}, it readily implies that
\begin{equation}\label{eqn10}
moex(\l)=\l_{o_{j_0+1}}+2=moex(f(\l)).
\end{equation}
For example, $\lambda=(5,4,1)\vdash 10$ with $moex(\lambda)=3$ and $moex(f(\lambda))=3$ with $j_0=1$. Also for $\lambda=(6,5,1)\vdash 12$ with $moex(\lambda)=3$ and $moex(f(\lambda))=3$ with $j_0=1$. Next, for $1\le j \le 5$, define
\begin{equation}\label{imageset}
f\left(\mathcal{P}^{[j]}_{od}(n)\right)=:\widetilde{\mathcal{P}}^{[j]}_{od}(n+1)\subset \mathcal{P}_{od}(n+1),
\end{equation}
and by \eqref{map}, $f|_{\mathcal{P}^{[j]}_{od}(n)}$ is bijective. Therefore,
\begingroup
\allowdisplaybreaks
\begin{align}\label{eqn11}
&\s_{od}\text{moex}(n)=\sum_{\l\in\mathcal{P}_{od}(n)}\!\!\!moex(\l)=\sum_{j=1}^{5}\!\sum_{\l\in\mathcal{P}^{[j]}_{od}(n)}\!\!moex(\l)\!=\!\!\!\!\sum_{\l\in\mathcal{P}^{[1]}_{od}(n)}moex(\l)+\sum_{j=2}^{5}\sum_{\l\in\mathcal{P}^{[j]}_{od}(n)}\!\!\!moex(\l)\nonumber\\
&\le \!\!\sum_{\l\in\mathcal{P}^{[1]}_{od}(n)}\!\!moex(\l)\!+\!\sum_{j=2}^{5}\!\sum_{\l\in\mathcal{P}^{[j]}_{od}(n)}\!\!\!\!moex(f(\l))=\!\!\!\sum_{\l\in\mathcal{P}^{[1]}_{od}(n)}\!\!\!\!moex(\l)\!+\!\sum_{j=2}^{5}\sum_{\widetilde{\l}\in\widetilde{\mathcal{P}}^{[j]}_{od}(n+1)}\!\!\!moex(\widetilde{\l}),
\end{align}
\endgroup
using \eqref{decomp} and \eqref{disjoint} in the second step, using \eqref{eqn8} and \eqref{eqn10} in the penultimate step, and using \eqref{imageset} and $f|_{\mathcal{P}^{[i]}_{od}(n)}$ is bijection in the final step. For $\l\in \mathcal{P}^{[1]}_{od}(n)$, by \eqref{eqn3} and \eqref{map}, we have
$moex(\l)=\l_{o_1}+2\ \text{and}\ moex(f(\l))=\l_{o_1}$. This consequently implies that $moex(\l)=moex(f(\l))+2$ and so
\begin{equation}\label{eqn12}
\sum_{\l\in\mathcal{P}^{[1]}_{od}(n)}moex(\l)=\sum_{\widetilde{\l}\in\widetilde{\mathcal{P}}^{[1]}_{od}(n+1)}moex(\widetilde{\l})+2\left|\mathcal{P}^{[1]}_{od}(n)\right|.
\end{equation}
For example, $\lambda=(5,4,3,2,2,1)\vdash 17$ with $moex(\lambda)=7$ and $moex(f(\lambda))=5$. From \eqref{eqn11} and \eqref{eqn12}, it yields that
\begin{equation}\label{eqn13}
\s_{od}\text{moex}(n)\le \sum_{i=1}^{5}\sum_{\widetilde{\l}\in\widetilde{\mathcal{P}}^{[i]}_{od}(n+1)}moex(\widetilde{\l})+2\left|\mathcal{P}^{[1]}_{od}(n)\right|.
\end{equation}
Now in order to prove the claimed inequality in the proposition, we now aim to find a non-empty set $Y:=Y(n+1)\subset \mathcal{P}_{od}(n+1)$ such that
\begin{equation}\label{eqn14}
Y\cap f\left(\mathcal{P}_{od}(n)\right)=\emptyset,\ \left|Y\right|\ge \left|\mathcal{P}^{[1]}_{od}(n)\right|,\ \text{and}\ moex(\mu)\ge 3\ \text{for}\ \mu\in Y.
\end{equation}
Because then we have
\begingroup
\allowdisplaybreaks
\begin{align*}
&\s_{od}moex(n+1)=\sum_{\l\in\mathcal{P}_{od}(n+1)}moex(\l)\ge \sum_{\l\in f\left(\mathcal{P}_{od}(n)\right)}moex(\l)+\sum_{\l\in Y}moex(\l)\\
&=\sum_{i=1}^{5}\sum_{\widetilde{\l}\in \mathcal{P}^{[i]}_{od}(n)}moex(\widetilde{\l})+\sum_{\l\in Y}moex(\l)\ge \sum_{i=1}^{5}\sum_{\widetilde{\l}\in\mathcal{P}^{[i]}_{od}(n)}moex(\widetilde{\l})+3|Y|\\
&\ge \s_{od}moex(n)-2\left|\mathcal{P}^{[1]}_{od}(n)\right|+3|Y|\ge \s_{od}moex(n)-2\left|Y\right|+3|Y|>\s_{od}moex(n),
\end{align*}
\endgroup
using $Y\cup f(\mathcal{P}_{od}(n))\subset\mathcal{P}_{od}(n+1)$ and \eqref{eqn14} in the second step, \eqref{eqn14} in the fourth step, \eqref{eqn13} in the fifth step, \eqref{eqn14} in the penultimate step, and using $Y\neq \emptyset$ in the final step. We now move on to construct such a set $Y\subset \mathcal{P}_{od}(n+1)$ depending on whether $n$ is a perfect square or not. First, note that if $n$ even and a perfect square, then $n=(2m)^2$ for some $m\in \mathbb{N}$. Note that for $\l\in \mathcal{P}^{[1]}_{od}(n)$, $\ell(\l_{od})\equiv 0\ (\text{mod}\ 2)$. For $m\in \mathbb{N}$, define
$\widehat{c}(n,m):=\#\left\{\l\vdash n: \l_j\equiv 0\ (\text{mod}\ 2)\ \text{and}\ \l_j<m\right\}$,
$a(n,m):=\frac nm-1=4m-1$, and $\overline{a}(n,m):=\frac{a(n,m)-1}{2}+1=2m$.
We have in total $\sum_{j=1}^{\frac{\overline{a}(n,m)}{2}}\#(\lambda_j(n,m))$ numbers of such partitions in $\mathcal{P}^{[1]}_{od}(n)$, which are as follows: $\lambda_1(n,m):=(a(n,m), a(n,m)-2,a(n,m)-4,\cdots,a(n,m)-(a(n,m)-1))$ and we see that exactly there is $\#(\lambda_1(n,m)):=1$ such partition. Set
$\lambda_2(n,m):=(a(n,m)-4,a(n,m)-6,\cdots,a(n,m)-(a(n,m)-1))\cup \lambda_{er}$ and thus, we obtain $\#(\lambda_2(n,m)):=\widehat{c}(a(n,m)+a(n,m)-2,a(n,m)-4)$ number of such partitions. Continuing in such a manner, we have
$\lambda_{\frac{\overline{a}(n,m)}{2}}(n,m):=\left(a(n,m)-(a(n,m)-3),a(n,m)-(a(n,m)-1)\right)\cup \lambda_{er}$ and exactly $\#(\lambda_{\frac{\overline{a}(n,m)}{2}}(n,m)):=\widehat{c}(a(n,m)+\cdots+a(n,m)-(a(n,m)-5),a(n,m)-(a(n,m)-3))$ number of such partitions. Each of the above partitions with $moex$ statistics are given by $a(n,m)+2$, $a(n,m)-2,\cdots,a(n,m)-(a(n,m)-3)+2$ along with their number of occurences respectively, is reduced by 2 in their image (under the map $f$), i.e., 
$$\sum_{\lambda \in \mathcal{P}^{[1]}_{od}(n)}moex(\l)=2\sum_{j=1}^{\frac{\overline{a}(n,m)}{2}}\#(\lambda_j(n,m))+\sum_{\lambda \in f(\mathcal{P}^{[1]}_{od}(n))}moex(\l).$$
For example, $16=(2\cdot2)^2$. Therefore, $a(16,2)=7$, which is the largest odd part can appear in the partitions of $\mathcal{P}^{[1]}_{od}(16)$ and $\overline{a}(16,2)=4$. Here we have all such partitions of $\mathcal{P}^{[1]}_{od}(16)$: $\lambda_1(16,2)=(7,5,3,1)$ and $\#(\lambda_1(16,2))=1$, $\lambda_2(16,2)=(3,1)\cup \lambda_{er}$ and $\#(\lambda_2(16,2))=\widehat{c}(12,3)$, where $\widehat{c}(12,3)$ counted by $\l_{er}$. Thus, 
$\sum_{\lambda \in \mathcal{P}^{[1]}_{od}(16)}moex(\l)=2\sum_{j=1}^{\frac{\overline{a}(16,2)}{2}}\!\#(\lambda_j(16,2))+\sum_{\lambda \in f(\mathcal{P}^{[1]}_{od}(16))}moex(\l)$. 
There always exist $\sum_{j=1}^{\frac{\overline{a}(n,m)}{2}}\#(\overline{\lambda}_j(n+1,m))$ number of partitions of $n+1$ in $\mathcal{P}_{od}(n+1)\setminus f(\mathcal{P}_{od}(n))$ each with $moex(\l)\ge 3$, which are as follows: $\overline{\lambda}_1(n+1,m):=(a(n,m)-1, a(n,m)-1,a(n,m)-5,a(n,m)-5,\cdots,a(n,m)-(a(n,m)-2),a(n,m)-(a(n,m)-2))\cup(1)$ and exactly $\#(\overline{\lambda}_1(n+1,m):=1$ such partition,
$\overline{\lambda}_2(n+1,m):=(a(n,m)-5,a(n,m)-5,\cdots,a(n,m)-(a(n,m)-2),a(n,m)-(a(n,m)-2))\cup(1)\cup\lambda_{er}$ and exactly $\widehat{c}(a(n,m)+(a(n,m)-2),a(n,m)-4)=:\#(\overline{\lambda}_2(n+1,m))$ number of such partitions, where $\lambda_{er}$ is enumerated by $\#(\overline{\lambda}_2(n+1,m))$. Continuing altogether, we finally get $\overline{\lambda}_{\frac{\overline{a}(n,m)}{2}}(n+1,m):=(a(n,m)-(a(n,m)-2),a(n,m)-(a(n,m)-2))\cup(1)\cup\lambda_{er}$ and exactly $\widehat{c}(a(n,m)+\cdots+a(n,m)-(a(n,m)-5),a(n,m)-(a(n,m)-3)):=\#(\overline{\lambda}_{\frac{\overline{a}(n,m)}{2}}(n+1,m))$ number of such partitions. It is immediate that $$\left|\mathcal{P}^{[1]}_{od}(n)\right|=\sum_{j=1}^{\frac{\overline{a}(n,m)}{2}}\#(\overline{\lambda}_j(n+1,m))=\sum_{j=1}^{\frac{\overline{a}(n,m)}{2}}\#(\lambda_j(n,m)):=|Y|.$$ 
For example, $16=(2\cdot2)^2$. Therefore, $a(16,2)=7$ and $\overline{a}(16,2)=4$. Here we have all such partitions of $\mathcal{P}_{od}(17)$, explicitly, $\lambda_1(17,2)=(6,6,2,2)\cup(1)$ and $\#(\lambda_1(17,2))=1$,\\
$\lambda_2(17,2)=(2,2)\cup(1)\cup \lambda_{er}$ and $\#(\lambda_2(17,2))=\widehat{c}(12,3)$ where $\lambda_{er}$ enumerated by $\widehat{c}(12,3)$.\\
It is clear that $\sum_{j=1}^{\frac{\overline{a}(16,2)}{2}}\#(\overline{\lambda}_j(17,2))=\sum_{j=1}^{\frac{\overline{a}(16,2)}{2}}\#(\lambda_j(16,2))$. Now, we consider $n$ is even but $n\neq (2m)^2$ for $m\in \mathbb{N}$. Set 
$n_1:=\underset{k<n}{\max}\left\{k: k=4\ell^2\ \text{for some} \ \ell \in \mathbb{N}\right\}$.
There are exactly $\sum_{j=1}^{\frac{\overline{a}(n_1,m_1)}{2}}\#(\lambda_j(n_1,m_1))$ numbers of such partitions of $n$, where $m_1$ is determined by $n_1=4m^2_1$.  For example, $12$ is even but not of the form $4m^2$ for $m \in \mathbb{N}$ and in this case, $n_1=4$.  When $n\neq 4m^2$, the same construction works for $n_1$ which is the maximum such even integer with $n_1<n$ and $n_1=4m^2$ for some $m \in \mathbb{N}$. Now, we consider $n$ is odd and $n=(2m+1)^2$ for some $m\in \mathbb{N}$. Define $b(n,m):=\frac{n+m}{m+1}=4m+1$ and $\overline{b}(n,m):=\frac{b(n,m)-1}{2}=2m$. In this case, we have in total $\sum_{j=1}^{\frac{\overline{b}(n,m)}{2}}\#(\lambda'_j(n,m))$ number of such partitions in $\mathcal{P}^{[1]}_{od}(n)$, explicitly written as: $\lambda'_1(n,m):=(b(n,m), b(n,m)-2,b(n,m)-4,\cdots,b(n,m)-(b(n,m)-1))$ and exactly $ \#(\lambda'_1(n,m)):=1$ such partition,
$\lambda'_2(n,m):=(b(n,m)-4,b(n,m)-6,\cdots,b(n,m)-(b(n,m)-1))\cup \lambda_{er}$ and exactly $\widehat{c}(b(n,m)+(b(n,m)-2),b(n,m)-4)=:\#(\lambda'_2(n,m))$ number of such partitions. Continuing in this fashion, our final set of partitions are $\lambda'_{\frac{\overline{b}(n,m)}{2}}(n,m):=(b(n,m)-(b(n,m)-5),b(n,m)-(b(n,m)-3),b(n,m)-(b(n,m)-1))\cup\lambda_{er}$ and obtain exactly $\widehat{c}(b(n,m)+\cdots+b(n,m)-(b(n,m)-7),b(n,m)-(b(n,m)-3)):=\#(\lambda'_{\frac{\overline{b}(n,m)}{2}}(n,m))$ number of such partitions.  Each of the above partitions with $moex$ is $b(n,m)+2$, $b(n,m)-2,\cdots,b(n,m)-(b(n,m)-5)+2$, counted with their number of occurences respectively, is reduced by $2$ in their image, i.e., $$\sum_{\lambda' \in \mathcal{P}^{[1]}_{od}(n)}moex(\l')=2\sum_{j=1}^{\frac{\overline{b}(n,m)}{2}}\#(\lambda'_j(n,m))+\sum_{\lambda' \in f(\mathcal{P}^{[1]}_{od}(n))}moex(\l').$$ For example, $25=(2\cdot2+1)^2$. Therefore, $b(25,2)=9$, which is the largest odd part that can appear in partitions belongs to $\mathcal{P}^{[1]}_{od}(25)$ and $\overline{b}(25,2)=4$. Here we have all such partitions belong to $\mathcal{P}^{[1]}_{od}(25)$ as follows:
$\lambda'_1(25,2)=(9,7,5,3,1)$ and $\#(\lambda'_1(25,2))=1$,
$\lambda'_2(25,2)=(5,3,1)\cup\lambda_{er}$ and $\#(\lambda'_2(25,2))=\widehat{c}(16,5)$, where $\lambda_{er}$ enumerated by $\widehat{c}(16,5)$.
Thus 
$\sum_{\lambda' \in \mathcal{P}^{[1]}_{od}(25)}moex(\l')=2\sum_{j=1}^{\frac{\overline{b}(25,2)}{2}}\#(\lambda'_j(25,2))+\sum_{\lambda' \in f(\mathcal{P}^{[1]}_{od}(25))}moex(\l')$. 
There always exist $\sum_{j=1}^{\frac{\overline{b}(n,m)}{2}}\#(\overline{\lambda}'_j(n+1,m))$ number of partitions of $n+1$ in $\mathcal{P}_{od}(n+1)\setminus f(\mathcal{P}_{od}(n))$ each with $moex\ge 3$, which are as follows: $\overline{\lambda}'_1(n+1,m):=(b(n,m)-1, b(n,m)-1,b(n,m)-5,b(n,m)-5,\cdots,b(n,m)-(b(n,m)-4),b(n,m)-(b(n,m)-2))\cup (b(n,m)-(b(n,m)-3),1)$ and exactly $\#(\overline{\lambda}'_1(n+1,m)):=1$ such partition, $\overline{\lambda}'_2(n+1,m):=(b(n,m)-5,b(n,m)-5,\cdots,b(n,m)-(b(n,m)-4),b(n,m)-(b(n,m)-2))\cup(b(n,m)-(b(n,m)-3),1)\cup\lambda'_{er}$ and exactly $\widehat{c}(b(n,m)+(b(n,m)-2),b(n,m)-4)=:\#(\overline{\lambda}'_2(n+1,m))$ number of such partitions, where $\lambda'_{er}$ corresponds to this $\text{mult}(\overline{\lambda}'_2(n+1,m))$ number of partitions. Finally, we have
$\overline{\lambda}'_{\frac{\overline{b}(n,m)}{2}}(n+1,m):=(b(n,m)-(b(n,m)-4),b(n,m)-(b(n,m)-2))\cup(b(n,m)-(b(n,m)-3),1)\cup\lambda'_{er}$ and exactly $\widehat{c}(b(n,m)+\cdots+b(n,m)-(b(n,m)-5),b(n,m)-(b(n,m)-3)):=\#(\overline{\lambda}'_{\frac{\overline{b}(n,m)}{2}}(n+1,m))$ number of such partitions. Therefore,
$$|\mathcal{P}^{[1]}_{od}(n)|=\sum_{j=1}^{\frac{\overline{b}(n,m)}{2}}\#(\overline{\lambda}'_j(n+1,m))=\sum_{j=1}^{\frac{\overline{b}(n,m)}{2}}\#(\lambda'_j(n,m))=:|Y|.$$ 
For example, $25=(2\cdot2+1)^2$. So, $b(25,2)=9$ and $\overline{b}(25,2)=4$. Here we have all such partitions of $\mathcal{P}_{od}(26)$: $\lambda'_1(26,2)=(8,8,4,2)\cup(3,1)$ and $\#(\lambda'_1(26,2))=1$,
$\lambda'_2(26,2)=(4,2)\cup(3,1)\cup\lambda'_{er}$ and $\#(\lambda'_2(26,2))=\widehat{c}(16,5)$ where $\lambda'_{er}$ enumerated by $\widehat{c}(16,5)$. Therefore, $\sum_{j=1}^{\frac{\overline{b}(25,2)}{2}}\#(\overline{\lambda}'_j(26,2))=\sum_{j=1}^{\frac{\overline{b}(25,2)}{2}}\#(\lambda'_j(25,2))$.
Finally for $n$ odd but not of the form $(2m+1)^2$ for some $m\in \mathbb{N}$, set 
$n_1':=\underset{k<n}{\max}\{k': k'=(2\ell'+1)^2\ \text{for some} \ \ell' \in \mathbb{N}\}$
and thus, there are exactly $\sum_{j=1}^{\frac{\overline{b}(n_1',m_1')}{2}}\#(\lambda'_j(n_1',m_1'))$ number of such partitions of $n$, where $n_1':=(2m_1'+1)^2$.  For example, $17$ is an odd integer which is not of the form $(2m+1)^2$ for $m \in \mathbb{N}$ and so $n_1'=9$.  When $n$ is not of the form $(2m+1)^2$, the analogous sketch works for $n_1'$ which is the maximum such odd integer with $n_1'<n$ and $n_1'=(2m'+1)^2$ for some $m' \in \mathbb{N}$. This concludes the proof.\qed

\emph{Proof of Theorem \ref{asymp}}: From Theorem \ref{thm_sigma}, we have  \begin{equation*}
\sum_{m=0}^{\infty}\sigma_{od}moex(n)q^n= \frac{(-q;q^2)_{\infty}}{(q^2;q^2)_{\infty}}\Bigl(1+\s^*(-q)\Bigr)=:A_1(q) B(q),
\end{equation*} where $A_1(q)$ be as in \eqref{A1} and $B(q)=1+\s^*(-q)$. Now using \eqref{cohensigma} with $q=e^{-t}$ and $t\to 0^{+}$, we get 
\begin{equation*}
\s^*(-e^{-t})=2\sum_{n=1}^{\infty}\frac{e^{-n^2 t}}{(-e^{-t};e^{-2t})_n}\to 2\sum_{n=1}^{\infty}\frac{1}{2^n}=2.
\end{equation*}
Thus $B(e^{-t})\sim 3$ as $t \rightarrow 0^{+}$ and so with \eqref{A1(q)asymp}, we get
$A_1(e^{-t})B(e^{-t})\sim 3\sqrt{\frac{t}{\pi}}e^{\frac{\pi^2}{8t}}$. 
 Note that by Theorem \ref{increasing}, we already know that $\s_{od}\text{moex}(n)$ is increasing. Thus, applying Proposition \ref{ingham} with $\alpha=\frac{3}{\sqrt{\pi}}$, $\beta=\frac 12$, and $C=\frac{\pi^2}{8}$, we have $\sigma_{od}\text{moex}(n) \sim \frac{3e^{\pi\sqrt{\frac{n}{2}}}}{4\sqrt{2}n}$. This concludes the proof of the theorem.\qed


\emph{Proof of Theorem \ref{thm_meex-4e}}: By definition of $p_{od,e}(n,j)$, we see that
\begin{equation*}
\sum_{n=0}^{\infty}\sum_{j=0}^{\infty}p_{od,e}(n,j)w^jq^n=(-q;q^2)_{\infty}\sum_{m=1}^{\infty}\prod_{t=1}^{m-1}\frac{q^{2t}}{1-q^{2t}}\prod_{t=m+1}^{\infty}\frac{1}{1-wq^{2t}},
\end{equation*}
where $(-q;q^2)_{\infty}$ enumerates the distinct odd parts, $\prod_{t=1}^{m-1}\frac{q^{2t}}{1-q^{2t}}$ counts (with multiplicity) the even parts $\le 2m-2$ which is the meex of the corresponding partition, and the factor $\prod_{t=m+1}^{\infty}\frac{1}{1-wq^{2t}}$ enumerates the even parts $> 2m$ and $w$ measures the length of these even parts. Thus, we have
\allowdisplaybreaks
\begingroup 
\begin{eqnarray}
&&\sum_{n=0}^{\infty}\sum_{j=0}^{\infty}p_{od,e}(n,j)w^jq^n 
= (-q;q^2)_{\infty}\sum_{m=1}^{\infty}\frac{q^{2+4+\cdots+(2m-2)}}{(q^2;q^2)_{m-1}(wq^{2m+2};q^2)_{\infty}}\nonumber \\ 
&&= (-q;q^2)_{\infty}\sum_{m=1}^{\infty}\frac{q^{2\binom{m}{2}}}{(q^2;q^2)_{m-1}(wq^{2m+2};q^2)_{\infty}}=\frac{(-q;q^2)_{\infty}}{(wq^4;q^2)_{\infty}}\sum_{m=1}^{\infty}\frac{q^{2\binom{m}{2}}(wq^4;q^2)_{m-1}}{(q^2;q^2)_{m-1}}\nonumber \\ 
&&= \frac{(-q;q^2)_{\infty}}{(wq^4;q^2)_{\infty}}\sum_{m=0}^{\infty}\frac{q^{2\binom{m+1}{2}}(wq^4;q^2)_{m}}{(q^2;q^2)_{m}}= \frac{(-q;q^2)_{\infty}}{(wq^4;q^2)_{\infty}} \lim_{A\rightarrow0}\frac{(-q^2/A;q^2)_{m}(wq^4;q^2)_{m}}{(0;q^2)_{m}(q^2;q^2)_{m}}A^m \nonumber \\ 
&&= \frac{(-q;q^2)_{\infty}}{(wq^4;q^2)_{\infty}}\lim_{A\rightarrow0} {}_{2}\phi_{1}\left(\begin{matrix}
-q^2/A,&wq^4\\&0
\end{matrix}\,;q^2,A\right) \nonumber\\ 
&&=  \frac{(-q;q^2)_{\infty}}{(wq^4;q^2)_{\infty}}\lim_{A\rightarrow0} \frac{(wq^4;q^2)_{\infty}(-q^2;q^2)_{\infty}}{(A;q^2)_{\infty}} {}_{2}\phi_{1}\left(\begin{matrix}
0,&A\\&-q^2
\end{matrix}\,;q^2,wq^4\right)\nonumber \\ 
&&=  (-q;q)_{\infty}{}_{2}\phi_{1}\left(\begin{matrix}
0,&0\\&-q^2
\end{matrix}\,;q^2,wq^4\right) = (-q;q)_{\infty}\sum_{m=0}^{\infty}\frac{w^mq^{4m}}{(-q^2;q^2)_m(q^2;q^2)_m}=(-q;q)_{\infty}\sum_{m=0}^{\infty}\frac{(wq^4)^m}{(q^4;q^4)_{m}} \nonumber \\ &&= \frac{(-q;q)_{\infty}}{(wq^4;q^4)_{\infty}}= \frac{1}{(q;q^2)_{\infty}(wq^4;q^4)_{\infty}}= \sum_{n=0}^{\infty}\sum_{j=0}^{\infty}p_{4e}(n,j)w^jq^n, \nonumber
\end{eqnarray}
\endgroup
using Lemma \ref{hypg_tran_formula} in the seventh step and Lemma \ref{Cauchy} in the penultimate step. This concludes the proof.\qed 

\section{Proof of Theorem \ref{thm_abarod} and Theorem \ref{thm_sigmabar}}\label{sec4}

\emph{Proof of Theorem \ref{thm_abarod}}: For $\l\in \mathcal{P}_{od}(n)$ with maximal excludant $2k-1$, we split $\l$ as given in \eqref{oddptn2} and \eqref{oddptn3}. We split $\l_{od}$ into two components $\l'_{od}$ and $\l''_{od}$, where the first component $\l'_{od}$ is a partition without repeated odd parts with parts $\le 2k-2$ and the second component $\l''_{od}$ is a partition without repeated odd parts with smallest part $2k$ and all even parts ($\ge 2k$) less than the largest part of $\l_{od}$ must appear (possibly with multiplicities). Taking the conjugation of the partition $\left( a(\l_{od}),\cdots, 2k+1,2k\right)$ where even parts (may) appear with multiplicities , we obtain a partition $\l$ with $\text{mult}(a(\l))=2k$ and all other parts of $\l$ are distinct. Thus, by \ref{def6} and Lemma \ref{Cauchy}, we have
\begingroup 
\allowdisplaybreaks
\begin{eqnarray}\label{eqn2}
&&\sum_{n=0}^{\infty}\overline{a}^{o}_{od}(n)q^n =  \Bigg[\sum_{k=1}^{\infty}\frac{z^{2k-1}(-q;q^2)_{k-1}}{(q^2;q^2)_{k-1}}\sum_{m=1}^{\infty}(-q,q)_{m-1}q^{2km} \Bigg]_{z=1}\nonumber\\
&&= \Biggl[\sum_{m=1}^{\infty}zq^{2m}(-q;q)_{m-1}\sum_{k=0}^{\infty}\frac{(-q;q^2)_k}{(q^2;q^2)_k}(z^2q^{2m})^k\Biggr]_{z=1}= \Biggl[\sum_{m=1}^{\infty}zq^{2m}(-q;q)_{m-1}\frac{(-z^2q^{2m+1};q^2)_{\infty}}{(z^2q^{2m};q^2)_{\infty}}\Biggr]_{z=1}\nonumber \\ &&= \sum_{m=1}^{\infty}q^{2m}(-q;q)_{m-1}\frac{(-q^{2m+1};q^2)_{\infty}}{(q^{2m};q^2)_{\infty}}= \frac{(-q;q^2)_{\infty}}{(q^2;q^2)_{\infty}}\sum_{m=1}^{\infty}q^{2m}(-q;q)_{m-1}\frac{(q^{2};q^2)_{m-1}}{(-q;q^2)_{m}} \nonumber \\ &&= \frac{(-q;q^2)_{\infty}}{(q^2;q^2)_{\infty}}\sum_{m=1}^{\infty}q^{2m}\frac{(q^{2};q^2)_{m-1}}{(q;q^2)_{m-1}(-q;q^2)_{m}}=\frac{q^2(-q;q^2)_{\infty}}{(1+q)(q^2;q^2)_{\infty}}\sum_{m=0}^{\infty}\frac{(q^{2};q^2)_{m}q^{2m}}{(q;q^2)_{m}(-q^3;q^2)_{m}}.
\end{eqnarray}
\endgroup 
Applying Theorem \ref{APT} with $q\mapsto q^2$, $a=-q^{-1}, b=q, B=q^2, A\to 0$, we see that
\begingroup
\allowdisplaybreaks
\begin{align*}
\sum_{m=0}^{\infty}\frac{(q^{2};q^2)_{m}q^{2m}}{(q;q^2)_{m}(-q^3;q^2)_{m}}&=\frac{q(q^2;q^2)_{\infty}}{(q;q^2)_{\infty}(-q^3;q^2)_{\infty}}\lim_{A\to 0}\sum_{m\ge 0}\frac{\left(\frac 1A;q^2\right)_m A^m(-q^4)^m}{(q^3;q^2)_{m+1}}\\
&\hspace{4 cm}+(1+q)\sum_{m\ge 0}\frac{(q;q^2)_{m+1}}{(q^3;q^2)_{m+1}}(-q)^m\\
&=\frac{q(q^2;q^2)_{\infty}}{(q;q^2)_{\infty}(-q^3;q^2)_{\infty}}\sum_{m\ge 0}\frac{q^{m^2+3m}}{(q^3;q^2)_{m+1}}+(1-q^2)\sum_{m\ge 0}\frac{(-q)^m}{1-q^{2m+3}}\\
&=\frac{q(q^2;q^2)_{\infty}}{(q;q^2)_{\infty}(-q^3;q^2)_{\infty}}\sum_{m\ge 0}\frac{q^{m^2+3m}}{(q^3;q^2)_{m+1}}+(1-q^2)\sum_{m\ge 0}(-q)^{m}\sum_{n\ge 0}q^{2mn+3n}\\
&=\frac{q(q^2;q^2)_{\infty}}{(q;q^2)_{\infty}(-q^3;q^2)_{\infty}}\sum_{m\ge 0}\frac{q^{m^2+3m}}{(q^3;q^2)_{m+1}}+(1-q^2)\sum_{n\ge 0}q^{3n}\sum_{m\ge 0}(-q^{2n+1})^{m}\\
&=\frac{q(q^2;q^2)_{\infty}}{(q;q^2)_{\infty}(-q^3;q^2)_{\infty}}\sum_{m\ge 0}\frac{q^{m^2+3m}}{(q^3;q^2)_{m+1}}+(1-q^2)\sum_{n\ge 0}\frac{q^{3n}}{1+q^{2n+1}}.
\end{align*}
\endgroup 
Combining this with \eqref{eqn2}, we obtain
\begin{align*}
\sum_{n=0}^{\infty}\overline{a}^{o}_{od}(n)q^n&=\frac{1}{(q;q^2)_{\infty}}\sum_{m\ge 0}\frac{q^{m^2+3m+3}}{(q^3;q^2)_{m+1}}+\frac{q^2(-q;q^2)_{\infty}}{(1+q)(q^4;q^2)_{\infty}}\sum_{m\ge 0}\frac{q^{3m}}{1+q^{2m+1}}\\
&=\frac{1}{(q;q^2)_{\infty}}\sum_{m\ge 1}\frac{q^{m^2+m+1}}{(q^3;q^2)_{m}}+\frac{q^2(-q;q^2)_{\infty}}{(1+q)(q^4;q^2)_{\infty}}\sum_{m\ge 0}\frac{q^{3m}}{1+q^{2m+1}}\\
&=\frac{q}{(q^3;q^2)_{\infty}}\sum_{m\ge 1}\frac{q^{m^2+m}}{(q;q^2)_{m+1}}+\frac{q^2(-q;q^2)_{\infty}}{(1+q)(q^4;q^2)_{\infty}}\sum_{m\ge 0}\frac{q^{3m}}{1+q^{2m+1}}\\
&=\frac{q}{(q^3;q^2)_{\infty}}\left(\sum_{m\ge 0}\frac{q^{m^2+m}}{(q;q^2)_{m+1}}-\frac{1}{1-q}\right)+\frac{q^2(-q;q^2)_{\infty}}{(1+q)(q^4;q^2)_{\infty}}\sum_{m\ge 0}\frac{q^{3m}}{1+q^{2m+1}}\\
&=\frac{q}{(q^3;q^2)_{\infty}}\left(\nu(-q)-\frac{1}{1-q}\right)+\frac{q^2(-q;q^2)_{\infty}}{(1+q)(q^4;q^2)_{\infty}}\sum_{m\ge 0}\frac{q^{3m}}{1+q^{2m+1}}\\
&=\frac{q}{(q;q)^2{\infty}}\left((1-q)\nu(-q)-1\right)+\frac{q^2(-q;q^2)_{\infty}}{(1+q)(q^4;q^2)_{\infty}}\sum_{m\ge 0}\frac{q^{3m}}{1+q^{2m+1}},
\end{align*}
which proves \eqref{thm_abarod1}. 

In a similar fashion, we have
\begingroup 
\allowdisplaybreaks
\begin{eqnarray}
\sum_{n=0}^{\infty}\overline{a}^{e}_{od}(n)q^n &=&   \sum_{m=1}^{\infty}(-q,q)_{m-1}q^{m}+\Bigg[\sum_{k=1}^{\infty}\frac{z^{2k}(-q;q^2)_{k}}{(q^2;q^2)_{k-1}}\sum_{m=1}^{\infty}(-q,q)_{m-1}q^{(2k+1)m} \Bigg]_{z=1}\nonumber\\
&=& \sum_{m=0}^{\infty}(-q,q)_{m}q^{m+1}+\Bigg[\sum_{k=1}^{\infty}\frac{z^{2k}(-q;q^2)_{k}}{(q^2;q^2)_{k-1}}\sum_{m=1}^{\infty}(-q,q)_{m-1}q^{(2k+1)m} \Bigg]_{z=1}\nonumber\\
&=&\sum_{m=0}^{\infty}(-q,q)_{m}q^{m+1}+\Biggl[\sum_{m=1}^{\infty}z^2q^{3m}(-q;q)_{m-1}\sum_{k=0}^{\infty}\frac{(-q;q^2)_k}{(q^2;q^2)_k}(z^2q^{2m})^k(1+q^{2k+1})\Biggr]_{z=1}\nonumber\\
&=&\sum_{m=0}^{\infty}(-q,q)_{m}q^{m+1}+\frac{(-q;q^2)_{\infty}}{(q^2;q^2)_{\infty}}\sum_{m=1}^{\infty}\frac{(-q;q)_{m-1}(q^{2};q^2)_{m-1}}{(-q;q^2)_{m}}q^{3m}\nonumber\\
&&\hspace{4 cm}+\frac{(-q;q^2)_{\infty}}{(q^2;q^2)_{\infty}}\sum_{m=1}^{\infty}\frac{(-q;q)_{m-1}(q^{2};q^2)_{m}}{(-q;q^2)_{m+1}}q^{3m+1},\nonumber 
\end{eqnarray}
\endgroup
which is \eqref{thm_abarod2}. This conclude the proof of the theorem.\qed 
 
\emph{Proof of Theorem \ref{thm_sigmabar}}:  For $\l\in \mathcal{P}_{od}(n)$ with $moex(\l)=2k-1$, we split $\l$ as given in \eqref{oddptn2} and \eqref{oddptn3}. We further decompose its odd component $\lambda_{od}$ into two components $\lambda_{od}'$ and $\lambda_{od}''$. The first component $\lambda_{od}'$ is a partition with only distinct odd parts with parts $\le 2k-3$ and the second component $\lambda_{od}''$ is a partition with only distinct odd parts with smallest part $2k+1$ and from smallest part $2k+1$ and other parts (up to the largest part of $\lambda_{od}$, i.e; $a(\lambda_{od})$) must appear. Now if we take the conjugation of $\lambda_{od}''$, it gives us a partition $\mu$ with $\text{mult}(a(\mu))=2k+1$ and multiplicity of remaining parts is $2$. Therefore, by \ref{def7},
\begin{eqnarray}
\sum_{n=0}^{\infty}\sigma_{od}\text{moax}(n)q^n&=&\frac{d}{dz}\Biggl[\frac{1}{(q^2;q^2)_{\infty}}\sum_{k=1}^{\infty}z^{2k-1}(-q;q^{2})_{k-1}\sum_{m=1}^{\infty}q^{2\cdot(1+2+\cdots(m-1))+(2k+1)m}\Biggr]_{z=1} \nonumber \\ &=& \frac{1}{(q^2;q^2)_{\infty}}\sum_{k=1}^{\infty}(2k-1)(-q;q^{2})_{k-1}\sum_{m=1}^{\infty}q^{m(m+2k)},\nonumber
\end{eqnarray} which proves (\ref{thm_sigmabar1}).

Next,  for $\l\in \mathcal{P}_{od}(n)$ with $meex(\l)=2k$, we split $\l$ as given in \eqref{oddptn2} and \eqref{oddptn3}. We can divide $\lambda_{er}$ into two components $\lambda_{er}'$ and $\lambda_{er}''$. The first component $\lambda_{er}'$ is a partition with even parts $\le 2k$ and the second component $\lambda_{er}''$ is a partition with only even parts and the smallest part $2k+2$ and from the smallest part $2k+2$ up to the largest part of $\lambda_{er}$, i.e., $a(\lambda_{er})$, in between all parts must appear. Now if we take the conjugation of $\lambda_{er}''$, it gives us a partition $\mu$ with  $\text{mult}(a(\mu))=2k+2$ and all parts from $1$ up to $a(\mu)-1$ if appear then exactly twice. Therefore, by \ref{def7},
\begingroup
\allowdisplaybreaks
\begin{align}\label{sigmamaexbasic}
&\sum_{n=0}^{\infty}\sigma_{od}\text{meax}(n)q^n=\frac{d}{dz}\Biggl[(-q;q^2)_{\infty}\sum_{k=1}^{\infty}\frac{z^{2k}}{(q^2;q^2)_{k-1}}\sum_{m=1}^{\infty}(-q^2;q^2)_{m-1}q^{(2k+2)m}\Biggr]_{z=1}\\
&= \frac{d}{dz}\Biggl[(-q;q^2)_{\infty}\sum_{m=1}^{\infty}z^2q^{4m}(-q^2;q^2)_{m-1}\sum_{k=0}^{\infty}\frac{(z^2q^{2m})^k}{(q^2;q^2)_k} \Biggr]_{z=1}\nonumber\\
&=
\frac{d}{dz}\Biggl[(-q;q^2)_{\infty}\!\sum_{m=1}^{\infty}\frac{z^2q^{4m}(-q^2;q^2)_{m-1}}{(z^2q^{2m};q^2)_{\infty}} \Biggr]_{z=1}\!\!\!\!\!\!\!\!=\!(-q;q^2)_{\infty}\sum_{m=1}^{\infty}q^{4m}(-q^2;q^2)_{m-1}\frac{d}{dz}\left(\frac{z^2}{(z^2q^{2m};q^2)_{\infty}}\right)_{z=1}.\nonumber
\end{align} 
\endgroup
Note that
\begingroup
\allowdisplaybreaks 
\begin{eqnarray}
&&\frac{d}{dz}\Biggl(\frac{z^2}{(z^2q^{2m};q^2)_{\infty}}\Biggr)_{z=1}= \Biggl[\frac{z^2}{(z^2q^{2m};q^2)_{\infty}} \frac{d}{dz} \log \Biggl(\frac{z^2}{(z^2q^{2m};q^2)_{\infty}}\Biggr)\Biggr]_{z=1}\nonumber \\ &&= \frac{1}{(q^{2m};q^2)_{\infty}}\frac{d}{dz} \Biggl[2\log z-\sum_{n=1}^{\infty}\log \Bigl(1-z^2q^{2n+2m}\Bigr)\Biggr]_{z=1}=\frac{2}{(q^{2m};q^2)_{\infty}}\Biggl[1+\sum_{n=1}^{\infty}\frac{q^{2n+2m}}{1-q^{2n+2m}}\Biggr]. \nonumber
\end{eqnarray} 
\endgroup

Applying this to \eqref{sigmamaexbasic}, it yields 
\begingroup
\allowdisplaybreaks
\begin{eqnarray}
\sum_{n=0}^{\infty}\sigma_{od}meax(n)q^n&=& 2(-q;q^2)_{\infty}\sum_{m=1}^{\infty}\frac{q^{4m}(-q^2;q^2)_{m-1}}{(q^{2m};q^2)_{\infty}}\Biggl[1+\sum_{n=1}^{\infty}\frac{q^{2n+2m}}{1-q^{2n+2m}}\Biggr]\nonumber \\ &=& 2\frac{(-q;q^2)_{\infty}}{(q^2;q^2)_{\infty}}\sum_{m=1}^{\infty}q^{4m}(-q^2;q^2)_{m-1}(q^2;q^2)_{m-1}\Biggl[1+\sum_{n=1}^{\infty}\frac{q^{2n+2m}}{1-q^{2n+2m}}\Biggr]\nonumber \\ &=& 2\frac{(-q;q^2)_{\infty}}{(q^2;q^2)_{\infty}}\sum_{m=1}^{\infty}q^{4m}(q^4;q^4)_{m-1}\Biggl[1+\sum_{n=m}^{\infty}\frac{q^{2(n+1)}}{1-q^{2(n+1)}}\Biggr]\nonumber \\ &=& 2\frac{(-q;q^2)_{\infty}}{(q^2;q^2)_{\infty}} \Biggl(1-(q^4;q^4)_{\infty}+\sum_{n=1}^{\infty}\frac{q^{2(n+1)}}{1-q^{2(n+1)}}\Bigl(1-(q^4;q^4)_n\Bigr)\Biggr), \nonumber
\end{eqnarray} 
\endgroup 
using the identity (can be proved by induction on $n$) $\sum_{m=1}^{n}q^{4m}(q^4;q^4)_{m-1}= 1- (q^4;q^4)_{n}$ in the fourth step. This proves \eqref{thm_sigmabar2} and concludes the proof.\qed

\section{Proof of Theorem \ref{evthm1} and Theorem \ref{evthm2}}\label{sec5}

In this section, we first prove Theorem \ref{evthm1} and then Theorem \ref{evthm2}.

\emph{Proof of Theorem \ref{evthm1}}: By \eqref{evneq1}, we get
\begin{align*}
&\sum_{n=0}^{\infty}a^{o}_{ed}(n)q^n=\sum_{n=0}^{\infty}\frac{q^{1+2+3+\dots+(2n-1)+2n}\displaystyle \prod_{m=n+1}^{\infty}\left(1+q^{2m}\right)\left(1-q^{2n+1}\right)}{\left(q;q^2\right)_{\infty}}\\
&=\frac{\left(-q^2;q^2\right)_{\infty}}{\left(q;q^2\right)_{\infty}}\sum_{n=0}^{\infty}\frac{q^{\binom{2n+1}{2}}\left(1-q^{2n+1}\right)}{\left(-q^2;q^2\right)_n}=\frac{\left(-q^2;q^2\right)_{\infty}}{\left(q;q^2\right)_{\infty}}\left(\sum_{n=0}^{\infty}\frac{q^{2n^2+n}}{\left(-q^2;q^2\right)_n}-\sum_{n=0}^{\infty}\frac{q^{(2n+1)(n+1)}}{\left(-q^2;q^2\right)_n}\right),
\end{align*}
which proves the first identity of Theorem \ref{evthm1}. Next, by \eqref{evneq1}, we obtain
\begin{align*}
\sum_{n=0}^{\infty}a^{e}_{ed}(n)q^n&=\sum_{n=0}^{\infty}\frac{q^{1+2+3+\dots+(2n-1)}\displaystyle \prod_{m=n+1}^{\infty}\left(1+q^{2m}\right)}{\left(q;q^2\right)_{\infty}}=\frac{\left(-q^2;q^2\right)_{\infty}}{\left(q;q^2\right)_{\infty}}\sum_{n=0}^{\infty}\frac{q^{2n^2-n}}{\left(-q^2;q^2\right)_n}.
\end{align*}
This concludes the proof.\qed 

\emph{Proof of Theorem \ref{evthm2}}: Using \eqref{eveqn2}, we have
\begin{align*}
&\sum_{n=0}^{\infty}\s_{ed}\text{moex}(n)q^n=\sum_{n=0}^{\infty}\frac{(2n+1)q^{1+3+5+\dots+(2n-1)}\left(-q^2;q^2\right)_{\infty}\left(1-q^{2n+1}\right) }{\left(q;q^2\right)_{\infty}}\\
&=\frac{\left(-q^2;q^2\right)_{\infty}}{\left(q;q^2\right)_{\infty}}\sum_{n=0}^{\infty}(2n+1)q^{n^2}\left(1-q^{2n+1}\right)=\frac{\left(-q^2;q^2\right)_{\infty}}{\left(q;q^2\right)_{\infty}}\left(\sum_{n=0}^{\infty}(2n+1)q^{n^2}-\sum_{n=1}^{\infty}(2n-1)q^{n^2}\right)\\
&=\frac{\left(-q^2;q^2\right)_{\infty}}{\left(q;q^2\right)_{\infty}}\left(1+2\sum_{n=0}^{\infty}q^{(n+1)^2}\right)=\frac{\left(-q^2;q^2\right)_{\infty}}{\left(q;q^2\right)_{\infty}}\left(2\sum_{n=1}^{\infty}q^{n^2}+1\right)=\frac{\left(-q^2;q^2\right)_{\infty}}{\left(q;q^2\right)_{\infty}}\sum_{n\in \mathbb{Z}}q^{n^2}\\
&=\frac{\left(-q^2;q^2\right)_{\infty}\left(q^2;q^2\right)_{\infty}\left(-q;q^2\right)^2_{\infty}}{\left(q;q^2\right)_{\infty}},
\end{align*}
using \eqref{Gauss1} in the last step and simplifying further, we finish the proof of the first identity of Theorem \ref{evthm2}. 

For the remaining identity, using \eqref{eveqn2}, we have
\begin{align*}
&\sum_{n=0}^{\infty}\s_{ed}\text{meex}(n)q^n=\sum_{n=0}^{\infty}\frac{2n\ q^{2+4+6+\dots+(2n-2)}\left(-q^{2n+2};q^2\right)_{\infty}}{\left(q;q^2\right)_{\infty}}=\frac{\left(-q^2;q^2\right)_{\infty}}{\left(q;q^2\right)_{\infty}}\sum_{n=0}^{\infty}\frac{2n\ q^{2\binom{n}{2}}}{\left(-q^2;q^2\right)_n}\\
&=\frac{\left(-q^2;q^2\right)_{\infty}}{\left(q;q^2\right)_{\infty}}\sum_{n=1}^{\infty}\frac{2n\ q^{2\binom{n}{2}}}{\left(-q^2;q^2\right)_n}=\frac{\left(-q^2;q^2\right)_{\infty}}{\left(q;q^2\right)_{\infty}}\sum_{n=1}^{\infty}\frac{2n\ q^{2\binom{n}{2}}}{\left(-q^2;q^2\right)_{n-1}}\left(1-\frac{q^{2n}}{1+q^{2n}}\right)\\
&=\frac{\left(-q^2;q^2\right)_{\infty}}{\left(q;q^2\right)_{\infty}}\left(\sum_{n=1}^{\infty}\frac{2n\ q^{n^2-n}}{\left(-q^2;q^2\right)_{n-1}}-\sum_{n=1}^{\infty}\frac{2n\ q^{n^2+n}}{\left(-q^2;q^2\right)_{n}}\right)\\
&=\frac{\left(-q^2;q^2\right)_{\infty}}{\left(q;q^2\right)_{\infty}}\left(\sum_{n=0}^{\infty}\frac{\left(2n+2\right)\ q^{n^2+n}}{\left(-q^2;q^2\right)_{n}}-\sum_{n=1}^{\infty}\frac{2n\ q^{n^2+n}}{\left(-q^2;q^2\right)_{n}}\right)=\frac{2\left(-q^2;q^2\right)_{\infty}}{\left(q;q^2\right)_{\infty}}\sum_{n=0}^{\infty}\frac{q^{n^2+n}}{\left(-q^2;q^2\right)_{n}}\\
&=\frac{2\left(-q^2;q^2\right)_{\infty}}{\left(q;q^2\right)_{\infty}}\s\left(q^2\right),
\end{align*} 
using \eqref{sigma} in the last step and we conclude the proof.\qed

\section{Concluding remarks}\label{con}
We conclude this paper by proposing the following problems with a brief description. We being with Theorem \ref{thm_aod}. We observe that the generating function involves $\frac{\left(-q;q^2\right)_{\infty}}{\left(q^2;q^2\right)_{\infty}}$ which is a modular form but we have failed to address the modularity of the remaining $q$-series. This leads to the following problem.
\begin{problem}\label{prob1}
What about the modularity of the $q$ hypergeometric series $\sum_{n\ge 0}\frac{q^{2n^2+n}}{\left(-q;q^2\right)_{n+1}}$?	
\end{problem}
Recently, asymptotics of Fourier coefficients of $v_1(q)$, a similar one as that of $v_2(q)$, studied by Folsom, Males, Rolen, and Storzer \cite{FMRS}. In Theorem \ref{asymp2}, we only obtained asymptotic main terms and thus propose the following.
\begin{problem}\label{prob2}
Is it possible to derive a Rademacher-type exact formula of $a^{o}_{od}(n)$ and $a^{e}_{od}(n)$?
\end{problem}
Next problem is to find a $q$-series proof for monotonicity of $a^{o}_{od}(n)$, $a^{e}_{od}(n)$, and $\s_{od}moex(n)$ because our proofs (see proof of Theorem \ref{increasing2} and Theorem \ref{increasing}) are combinatorial. For a power series $A(q)$, $A(q)\succeq 0$ denotes that $A(q)$ has non-negative coefficients.
\begin{problem}\label{prob3}
Show that
$$\left(1-q\right)\sum_{n=0}^{\infty}a^{o}_{od}(n)q^n\succeq 0,\  \left(1-q\right)\sum_{n=0}^{\infty}a^{e}_{od}(n)q^n\succeq 0,\ \text{and}\ \left(1-q\right)\sum_{n=0}^{\infty}\s_{od}\textnormal{moex}(n)q^n\succeq 0.$$
\end{problem}
Next, we note that the proof of Theorem \ref{thm_meex-4e} is based on $q$-series transformations. Thus we ask the following:
\begin{problem}\label{prob4}
Is there a combinatorial proof of Theorem \ref{thm_meex-4e}?	
\end{problem}
Regarding Theorems \ref{thm_abarod} and \ref{thm_sigmabar}, we raise the following question.
\begin{problem}\label{quest3}
Is there any simpler representation for the right-hand side of \eqref{thm_abarod2}, \eqref{thm_sigmabar1}, and \eqref{thm_sigmabar2}? 
	\end{problem}
Due to Zagier \cite[Lemma, page 5]{Zagier}, we know that $\sigma(q)=-\sigma^*(q^{-1})$ for every roots of unity. In this spirit with regards to Theorem \ref{thm_sigma} and Theorem \ref{evthm2}, we ask the following:
\begin{problem}\label{quest4}
 Are the generating functions of $\{\sigma_{od}\textnormal{moex}(n)\}_{n\ge 0}$ and $\{\sigma_{ed}\textnormal{meex}(n)\}_{n\ge 0}$ related?
\end{problem}
As a final remark, we would like to point out that it might be worthwhile to investigate maximal excludant for partitions without repeated even parts.

	\begin{center}
	\textbf{Acknowledgements}
\end{center}
The author sincerely thanks her mentor Prof. Brundaban Sahu for his valuable advice and suggestions. The author would also like to thank Dr. Bibekananda Maji and Dr. Pramod Eyyunni for fruitful discussions. This work is financially supported by the Council of Scientific and Industrial Research (CSIR File No: 09/1002(25470)/2025-EMR-I). Finally, the author thanks National Institute of Science Education and Research (NISER), Bhubaneswar for its hospitality and support. 

\begin{center}
	\textbf{Data Availability}
\end{center}
We hereby confirm that Data sharing not applicable to this article as no datasets were generated or analyzed during the current study.

\begin{center}
	\textbf{Ethics Declaration}
\end{center}
The author confirms that they have no conflict of interest in connection with this paper.

\end{document}